\documentclass[12pt]{amsart}

\usepackage{amssymb, amsmath, amsthm}

\allowdisplaybreaks

\numberwithin{equation}{section}

\theoremstyle{plain}
\newtheorem{thm}{Theorem}[section]
\newtheorem{prop}[thm]{Proposition}
\newtheorem{lem}[thm]{Lemma}

\theoremstyle{definition}
\newtheorem{defn}[thm]{Definition}

\newtheorem{example}[thm]{Example}

\newcommand{\ichi}{\mathbf{1}}

\newcommand{\K}{\mathbb{K}}
\newcommand{\N}{\mathbb{N}}
\newcommand{\R}{\mathbb{R}}
\newcommand{\T}{\mathbb{T}}
\newcommand{\Z}{\mathbb{Z}}
\newcommand{\calA}{\mathcal{A}}
\newcommand{\calB}{\mathcal{B}}

\newcommand{\calF}{\mathcal{F}}

\newcommand{\calM}{\mathcal{M}}
\newcommand{\calS}{\mathcal{S}}
\newcommand{\supp}{\mathrm{supp}\, }

\newcommand{\II}{I\hspace{-3pt}I}
\newcommand{\op}{\mathrm{Op}}
\newcommand{\ZnN}{(\mathbb{Z}^n)^N}

\begin{document}
\title[Multilinear pseudo-differential operators]
{
Boundedness of 
multilinear pseudo-differential operators 
of $S_{0,0}$-type in $L^2$-based amalgam spaces 
}

\author[T. Kato]{Tomoya Kato}
\author[A. Miyachi]{Akihiko Miyachi}
\author[N. Tomita]{Naohito Tomita}

\address[T. Kato]
{Division of Pure and Applied Science, 
Faculty of Science and Technology, Gunma University, 
Kiryu, Gunma 376-8515, Japan}
\address[A. Miyachi]
{Department of Mathematics, 
Tokyo Woman's Christian University, 
Zempukuji, Suginami-ku, Tokyo 167-8585, Japan}
\address[N. Tomita]
{Department of Mathematics, 
Graduate School of Science, Osaka University, 
Toyonaka, Osaka 560-0043, Japan}

\email[T. Kato]{t.katou@gunma-u.ac.jp}
\email[A. Miyachi]{miyachi@lab.twcu.ac.jp}
\email[N. Tomita]{tomita@math.sci.osaka-u.ac.jp}

\date{\today}

\keywords{Multilinear pseudo-differential operators,
multilinear H\"ormander symbol classes}
\thanks{This work was supported by JSPS KAKENHI Grant Numbers 
JP17J00359 (Kato), JP16H03943 (Miyachi), and JP16K05201 (Tomita).}

\subjclass[2010]{35S05, 42B15, 42B35}

\begin{abstract}
We consider the multilinear 
pseudo-differential operators 
with symbols in a generalized $S_{0,0}$-type class 
and prove the boundedness of 
the operators from 
$(L^2,\ell^{q_1}) \times \dots \times (L^2,\ell^{q_N})$
to $(L^2,\ell^{r})$, where $(L^2, \ell^{q})$ denotes 
the $L^2$-based amalgam space.  
This extends the previous result by the same 
authors, which treated the bilinear pseudo-differential operators and gave the $L^2 \times L^2 $ to 
$(L^2, \ell^{1})$ boundedness. 
\end{abstract}

\maketitle

\section{Introduction}\label{Introduction}

First of all, throughout this paper, the letter $N$ denotes 
a positive integer unless the contrary is explicitly stated.

For a bounded measurable 
function $\sigma = \sigma (x, \xi_1, \dots, \xi_N)$ on $(\R^n)^{N+1}$,
the multilinear pseudo-differential operator
$T_{\sigma}$ is defined by
\[
T_{\sigma}(f_1,\dots,f_N)(x)
=\frac{1}{(2\pi)^{Nn}}
\int_{(\R^n)^N}e^{i x \cdot(\xi_1+\cdots+\xi_N)}
\sigma (x, \xi_1, \dots, \xi_N) 
\prod_{j=1}^N\widehat{f_j}(\xi_j)
\, d\xi_1 \dots d\xi_N
\]
for $f_1,\cdots,f_N \in \calS(\R^n)$. 
The function $\sigma$ is called the symbol 
of the operator $T_{\sigma}$. 
If $N=1$ (resp.\ $N=2$), we call 
$T_{\sigma}$ the linear (resp.\ bilinear)
 pseudo-differential operator.

For the boundedness of the multilinear operators 
$T_{\sigma}$, we shall 
use the following terminology. 
Let $X_{j}$, $j=1,\dots, N$, and $Y$ be 
function spaces on $\R^n$ 
equipped with quasi-norms 
$\|\cdot \|_{X_{j}}$ and $\|\cdot \|_{Y}$, 
respectively.  
If there exists a constant $A$ such that 
\begin{equation}\label{boundedness-X_1X_2Y}
\|T_{\sigma}(f_1,\dots,f_N)\|_{Y}
\le A \prod_{j=1}^{N} \|f_{j}\|_{X_{j}} 
\;\;
\text{for all}
\;\;
f_{j}\in \calS \cap X_{j},  
\;\;
j=1,\dots,N,
\end{equation}
then, 
with a slight abuse of terminology, 
we say that 
$T_{\sigma}$ is bounded from 
$X_1 \times \dots \times X_N$ to $Y$ 
and write 
$T_{\sigma}: X_1 \times \dots \times X_N \to Y$.  
The smallest constant $A$ in  
\eqref{boundedness-X_1X_2Y} 
is denoted by 
$\|T_{\sigma}\|_{X_1 \times \dots \times X_N \to Y}$. 
If $\calA$ is a class of symbols,  
we denote by $\mathrm{Op}(\calA)$
the class of all operators $T_{\sigma}$ 
corresponding to $\sigma \in \calA$. 
If $T_{\sigma}: X_1 \times \dots \times X_N \to Y$ 
for all $\sigma \in \calA$, 
then we write 
$
\mathrm{Op}(\calA) 
\subset B (X_1 \times \dots \times X_N \to Y)$. 

We introduce the following symbol class of $S_{0,0}$-type.

\begin{defn}\label{BSW00} 
For a nonnegative function 
$W$ on $(\R^n)^N$, 
we denote by 
$S^{W}_{0,0}(\R^n, N)$ the set of all those 
smooth functions 
$\sigma = \sigma (x, \xi_1, \dots, \xi_N)$ 
on $(\R^n)^{N+1}$ 
such that the estimate 
\[
|\partial^{\alpha}_x 
\partial^{\beta_1}_{\xi_1} \cdots \partial^{\beta_N}_{\xi_N} 
\sigma(x,\xi_1,\dots,\xi_N)|
\le C_{\alpha,\beta_1,\dots,\beta_N}
W (\xi_1,\dots,\xi_N)
\]
holds for all multi-indices
$\alpha, \beta_1, \dots, \beta_N \in \N_0^n
=\{ 0,1,2,\dots\}^{n}$. 
We shall call $W$ the {\it weight function\/} 
of the class $S^{W}_{0,0}(\R^n, N)$. 
\end{defn}

For the weight function 
$W(\xi_1,\dots,\xi_N) = (1+|\xi_1|+\dots+|\xi_N|)^{m}$, 
$m\in\R$, 
we denote
the class $S_{0,0}^{W}(\R^n, N)$ 
by $S_{0,0}^{\langle m \rangle } (\R^n, N)$. 
When $N=1$, the 
class $S_{0,0}^{\langle m \rangle } (\R^n, 1)$ 
coincides with the well-known H\"ormander class 
$S_{0,0}^{m}(\R^n)$. 
In several literatures, the class 
$S_{0,0}^{\langle m \rangle } (\R^n, N)$ 
is denoted by several different symbols 
and is called the multilinear H\"ormander class.

The study of boundedness of 
the pseudo-differential operators with symbols in 
linear and multilinear H\"ormander classes
have been done by a lot of researchers.
For instance, in the linear case, $N=1$,
the celebrated Calder\'on-Vaillancourt theorem 
states that
\begin{equation}\label{linearLpbdd}
\op(S_{0,0}^{\langle 0 \rangle }(\R^{n}, 1)) 
\subset B(L^{2} \to L^{2}) 
\end{equation}
(see \cite{CV}). 
In the bilinear case, $N=2$,
in the paper \cite{MT-2013}, 
the second and the third named authors 
of the present paper 
proved that 
\begin{equation}\label{bilinearL2L2h1bdd}
\op(S_{0,0}^{\langle -n/2 \rangle }(\R^{n},2)) 
\subset B(L^{2} \times L^{2} \to h^{1}) 
\end{equation}
and 
\begin{equation}\label{bilinearL2bmoL2bdd}
\op(S_{0,0}^{\langle -n/2 \rangle }(\R^{n},2)) 
\subset B(L^{2} \times bmo \to L^{2}), 
\end{equation}
where $h^1$ is the local Hardy space and 
$bmo$ is the local $BMO$ space. 
By interpolation, these imply that  
\begin{equation}\label{bilinearLpbdd}
\op(S_{0,0}^{\langle -n/2 \rangle }(\R^{n},2)) 
\subset B(L^{q_1} \times L^{q_2} \to L^{r})
\end{equation}
for $1 \leq r \leq 2 \leq q_1,q_2 \leq \infty $ with $1/q_1+1/q_2=1/r$. 
In \cite{MT-2013}, 
it is also proved 
that $m=-n/2$ is the critical number in 
\eqref{bilinearL2L2h1bdd}, 
\eqref{bilinearL2bmoL2bdd}, 
and 
\eqref{bilinearLpbdd}. 
In the $N$-fold multilinear case, $N \geq 3$,
Michalowski--Rule--Staubach \cite[Theorem 3.3]{MRS-2014}
proved that
\begin{equation}\label{multilinearLpbdd}
\op(S_{0,0}^{\langle m\rangle }(\R^{n},N)) 
\subset B(L^{q_1} \times \dots \times L^{q_N} \to L^{r})
\quad\textrm{with}\quad
m<-nN/2
\end{equation}
holds for $2 \leq q_1,\dots, q_N \leq \infty$ and $2/N \leq r \leq \infty$ with $1/q_1+\dots+1/q_N=1/r$.

Quite recently, the authors \cite{KMT-arXiv} 
introduced the 
symbol class $S_{0,0}^{W}(\R^n, 2)$ of 
Definition \ref{BSW00} and, 
under certain condition on the weight $W$,  
proved that the bilinear operators 
$T_{\sigma}$ with symbols in the class 
$S_{0,0}^{W}(\R^n, 2)$ 
are bounded from $L^2 \times L^2$ to the 
$L^2$-based amalgam space $(L^2,\ell^1)$.
For the definition of the amalgam spaces,
see Section \ref{secamalgam}.
This result improves \eqref{bilinearL2L2h1bdd}  
since 
the general symbol class  $S_{0,0}^{W}(\R^n, 2)$ 
of \cite{KMT-arXiv} can be wider than the bilinear 
H\"ormander class 
$S_{0,0}^{\langle -n/2 \rangle }(\R^n,2)$ and 
the target space 
$(L^2,\ell^1)$ is 
continuously embedded into $h^1$.

In the present paper, 
we will consider 
the multilinear pseudo-differential operators with 
symbols in the class $S_{0,0}^{W}(\R^n, N)$ 
and, under certain condition on the weight $W$, 
prove the boundedness of the corresponding 
operators from the product of the amalgam spaces
$(L^2,\ell^{q_1}) \times \dots \times (L^2,\ell^{q_N})$ 
to $(L^2,\ell^{r})$. 
Our result extends the result of \cite{KMT-arXiv}
in two ways. 
Firstly, we generalize the result for 
the bilinear pseudo-differential operators 
to the case of $N$-fold multilinear operators.
Secondly, 
we deal with the $L^2$-based amalgam spaces
not only as the target space but also as the domain space. 
Our result also generalizes or improves 
the results stated in
\eqref{linearLpbdd}--\eqref{multilinearLpbdd}; 
for this see the next to the last paragraph 
in this section.

Now we shall state our main result more precisely. 
We use the class of weight functions defined below. 
This is originally introduced in 
\cite[Definition 1.2]{KMT-arXiv} 
in the case $N=2$.
The definition below covers all $N\geq 1$.

\begin{defn}\label{classB}
We denote by 
$\calB_N (\ZnN)$ 
the set of 
all those nonnegative functions 
$V$ on $\ZnN$  
for which there exists a constant $c\in (0, \infty)$ 
such that the inequality 
\begin{equation}\label{BL222}
\sum_{\nu_1, \dots, \nu_N \in \Z^n} 
V(\nu_1, \dots, \nu_N) 
A_0(\nu_1+\dots+ \nu_N) 
\prod_{j=1}^N A_j(\nu_j)
\le c \prod_{j=0}^N
\|A_j\|_{\ell^2 (\Z^n) } 
\end{equation}
holds for all nonnegative 
functions $A_j$, $j=0,1,\dots,N$, on $\Z^n$. 
\end{defn}

By using the class above, 
the main theorem of this paper is
stated as follows. 

\begin{thm}\label{main-thm-1}
Let $V$ be a nonnegative bounded function 
on $\ZnN$ 
and let 
\[
\widetilde{V} (\xi_1, \dots, \xi_N)
=
\sum_{\nu_1, \dots, \nu_N \in \Z^n} 
V(\nu_1, \dots, \nu_N) 
\prod_{j=1}^N \ichi_{Q} (\xi_j - \nu_j) ,
\quad 
(\xi_1, \dots, \xi_N) \in (\R^n)^N,  
\]
where $Q=[-1/2, 1/2)^n$. 
Then the following hold. 
\\
$(1)$ 
If there exist $q_{j}, r \in (0, \infty]$, $j=1,\dots,N$, such that 
all $T_{\sigma} \in \mathrm{Op}
(S^{\widetilde{V}}_{0,0} (\R^n, N))$ 
are 
bounded from 
$(L^2,\ell^{q_1}) \times \dots \times (L^2,\ell^{q_N})$ 
to $(L^2,\ell^{r})$, 
then $V \in \calB_N (\ZnN)$. 
\\
$(2)$ 
Conversely, 
if $V \in \calB_N (\ZnN)$, 
then 
all 
$T_{\sigma} \in \mathrm{Op}
(S^{\widetilde{V}}_{0,0} (\R^n, N))$ 
are 
bounded from 
$(L^2,\ell^{q_1}) \times \dots \times (L^2,\ell^{q_N})$ 
to $(L^2,\ell^{r})$ 
for $q_{j}, r \in (0, \infty]$, $j=1,\dots,N$,
satisfying $1/q_1+\dots+1/q_N\ge 1/r$. 
\end{thm}

As we mentioned above, 
this theorem generalizes 
\cite[Theorem 1.3]{KMT-arXiv}, 
where the case 
$N=2$, $q_1=q_2=2$, and $r=1$ was considered.

Here are some typical examples of functions in 
the class $\calB_N (\ZnN)$.

\begin{example}\label{example-of-V}
The following functions $V$ on 
$\ZnN$ belong to 
the class $\calB_N (\ZnN)$: 
\begin{align}
&\label{V-n/2}
V(\nu_1, \dots, \nu_N)= (1 + |\nu_1|+ \cdots + |\nu_N|)^{-(N-1)n/2};   
\\
&\label{Vproduct2}
V(\nu_1, \dots, \nu_N) = 
\prod_{j=1}^N (1+ |\nu_j| )^{-a_{j}}, 
\quad 
0< a_j < n/2, \quad \sum_{j=1}^N a_j=(N-1)n/2;  
\end{align}
where 
$\nu_j  \in \Z^n$, 
$j=1, \dots, N$, 
and we assume $N \ge 2$ in \eqref{Vproduct2}. 
\end{example}

When $N=1$, the function of 
\eqref{V-n/2} is the constant function 
$V(\nu_1)=1$  
and for this weight function 
the class $S_{0,0}^{\widetilde V}(\R^n,1)$ 
is identical with $S_{0,0}^{\langle 0 \rangle } (\R^n, 1)$. 
The class $S_{0,0}^{\widetilde V}(\R^n, N)$ 
with $V$ of \eqref{V-n/2}
is equal to the class 
$S_{0,0}^{\langle -(N-1)n/2 \rangle } (\R^n,N)$.  
Notice that 
the class $S_{0,0}^{\widetilde V}(\R^n, N)$ 
with $V$ of \eqref{Vproduct2} 
is wider than 
$S_{0,0}^{\langle -(N-1)n/2 \rangle } (\R^n,N)$ 
if $N \geq 2$.
More generally, in Proposition \ref{weightLweak}, 
we will prove that all nonnegative 
functions in the Lorentz space 
$\ell^{\frac{2N}{N-1}, \infty}(\Z^{Nn})$
belong to $\calB_N (\ZnN)$. 
In Proposition \ref{productLweak}, 
we also prove that $\calB_N (\ZnN)$ contains
functions generalizing \eqref{Vproduct2}.

To conclude the overview of our result, 
we shall briefly give comments on comparison between
Theorem \ref{main-thm-1} 
and the results 
\eqref{linearLpbdd}--\eqref{multilinearLpbdd}. 
When $N=1$, the boundedness of 
Theorem \ref{main-thm-1} (2) 
for $V(\nu_1)=1$ and $q_1=r=2$ 
coincides with 
\eqref{linearLpbdd}, 
since $(L^2,\ell^2)=L^2$. 
When $N \geq 2$, 
Theorem \ref{main-thm-1} (2) 
for $V$ of \eqref{V-n/2}, 
combined with 
the embeddings 
$L^{q} \hookrightarrow (L^2,\ell^q)$
for $2 \le q \le \infty$, 
$bmo \hookrightarrow (L^2,\ell^{\infty})$,
and $(L^2,\ell^r) \hookrightarrow h^r$ for $0 < r \leq 2$ 
(see Section \ref{secamalgam}), 
implies that
\begin{equation*}
\op
(S^{\langle  - (N-1)n/2 \rangle }_{0,0} (\R^n, N))
\subset B(L^{q_1} \times \dots \times 
L^{q_N} \to h^{r})
\end{equation*}
for 
$2/N \leq r \leq 2 \leq q_{1},\dots,q_{N} \leq \infty $
and $1/q_1+\cdots+1/q_{N}=1/r$,
where $L^{q_j}$ can be replaced by 
$bmo$ for $q_j=\infty$. 
Thus, under these conditions of $q_j$ and $r$,
Theorem \ref{main-thm-1}
covers 
\eqref{bilinearL2L2h1bdd}, 
\eqref{bilinearL2bmoL2bdd}, \eqref{bilinearLpbdd}, 
and
\eqref{multilinearLpbdd}.

We end this section by mentioning the plan of this paper.
In Section \ref{section2}, 
we will give the basic notations used throughout this paper 
and recall the definitions and properties of some function spaces.
In Section \ref{sectionWeight}, 
we give several properties of the class 
$\calB_N (\ZnN)$
and prove 
that it contains 
the functions $V$ of Example \ref{example-of-V}.  
After we prepare several lemmas in 
Section \ref{sectionLemma}, 
we prove Theorem \ref{main-thm-1} in Section \ref{sectionMain}. 
In Section \ref{sectionMain}, we also 
give a theorem, Theorem \ref{main-thm-2}, 
which is concerned with 
symbols satisfying low regularity condition. 
In Section \ref{sectionBddinLp},
we show another theorem, 
Theorem \ref{main-thm-3},
which is concerned with 
symbols with low regularity 
and boundedness of the operators 
in the $L^p$ framework.
In Section \ref{sectionSharpness},
we prove the sharpness of Theorem \ref{main-thm-3}.


\section{Preliminaries}\label{section2}

\subsection{Basic notations} 

We collect notations which will be used throughout this paper.
We denote by $\R$, $\Z$, $\N$, and $\N_0$
the sets of real numbers, integers, positive integers, 
and nonnegative integers, respectively. 
We denote by $Q$ the $n$-dimensional 
unit cube $[-1/2,1/2)^n$.
The cubes $\tau + Q$, $\tau \in \Z^n$, are mutually 
disjoint and constitute a partition of the 
Euclidean space $\R^n$.
This implies integral of a function on $\R^{n}$ 
can be written as 
\begin{equation}\label{cubicdiscretization}
\int_{\R^n} f(x)\, dx = 
\sum_{\tau \in \Z^{n}} \int_{Q} f(x+ \tau)\, dx. 
\end{equation}
We denote by $B_R$ the closed ball in $\R^n$
of radius $R>0$ centered at the origin.
For $1 \leq p \leq \infty$, $p^\prime$ is 
the conjugate number of $p$ defined by 
$1/p + 1/p^\prime =1$.
We write 
$[s] = \max\{ n \in \Z : n \leq s \}$ for $s \in \R$. 
For $x\in \R^d$, we write 
$\langle x \rangle = (1 + | x |^2)^{1/2}$. 
Thus, for $N\geq1$, $\langle (x_1,\dots,x_N) \rangle = (1+|x_1|^2+\dots+|x_N|^2)^{1/2}$ 
for $ (x_1,\dots,x_N) \in (\R^n)^N$.

For two nonnegative functions $A(x)$ and $B(x)$ defined 
on a set $X$, 
we write $A(x) \lesssim B(x)$ for $x\in X$ to mean that 
there exists a positive constant $C$ such that 
$A(x) \le CB(x)$ for all $x\in X$. 
We often omit to mention the set $X$ when it is 
obviously recognized.  
Also $A(x) \approx B(x)$ means that
$A(x) \lesssim B(x)$ and $B(x) \lesssim A(x)$.

We denote the Schwartz space of rapidly 
decreasing smooth functions
on $\R^d$ 
by $\calS (\R^d)$ 
and its dual,
the space of tempered distributions, 
by $\calS^\prime(\R^d)$. 
The Fourier transform and the inverse 
Fourier transform of $f \in \calS(\R^d)$ are given by
\begin{align*}
\mathcal{F} f  (\xi) 
&= \widehat {f} (\xi) 
= \int_{\R^d}  e^{-i x \cdot \xi} f(x) \, dx, 
\\
\mathcal{F}^{-1} f (x) 
&= \check f (x)
= \frac{1}{(2\pi)^d} \int_{\R^d}  e^{i x \cdot \xi } f( \xi ) \, d\xi,
\end{align*}
respectively. 
We also deal with the partial Fourier transform 
of a Schwartz function $f (x,\xi_{1},\dots,\xi_{N})$,  
$x,\xi_{1},\dots,\xi_{N} \in\R^{n}$. 
In this case, we denote the partial Fourier transform 
with respect to the $x$ and $\xi_{j}$ variables 
by $\calF_0$ and $\calF_{j}$, $j=1,\dots,N$, respectively. 
We also write the Fourier transform on 
$(\R^{n})^{N}$ for the $\xi_{1},\dots, \xi_{N}$ variables as
$\calF_{1,\dots,N} = \calF_1 \cdots \calF_{N}$.
For $m \in \calS^\prime (\R^n)$, 
the Fourier multiplier operator is defined by
\begin{equation*}
m(D) f 
=
\mathcal{F}^{-1} \left[m \cdot \mathcal{F} f \right].
\end{equation*}
We also use the notation $(m(D)f)(x)=m(D_x)f(x)$ 
when we indicate which variable is considered.

For a measurable subset $E \subset \R^d$, 
the Lebesgue space $L^p (E)$, $0<p\le \infty$, 
is the set of all those 
measurable functions $f$ on $E$ such that 
$\| f \|_{L^p(E)} = 
\left( \int_{E} \big| f(x) \big|^p \, dx \right)^{1/p} 
< \infty 
$
if $0< p < \infty$ 
or   
$\| f \|_{L^\infty (E)} 
= 
\operatorname{ess \, sup}_{x\in E} |f(x)|
< \infty$ if $p = \infty$. 
We also use the notation 
$\| f \|_{L^p(E)} = \| f(x) \|_{L^p_{x}(E)} $ 
when we want to indicate the variable explicitly.

The uniformly local $L^2$ space, denoted by 
$L^2_{ul} (\R^d)$, consists of 
all those measurable functions $f$ on 
$\R^d$ such that 
\begin{equation*}
\| f \|_{L^2_{ul} (\R^d) } 
= \sup_{\nu \in \Z^d}
\left( \int_{[-1/2, 1/2)^d} 
\big| f(x+\nu) 
\big|^2 \, dx 
\right)^{1/2}
< \infty 
\end{equation*}
(this notion can be found in 
\cite[Definition 2.3]{kato 1975}).

Let $\K$ be a countable set. 
We define the sequence spaces 
$\ell^q (\K)$ and $\ell^{q, \infty} (\K)$ 
as follows.  
The space $\ell^q (\K)$, $0< q \le \infty$, 
consists of all those 
complex sequences $a=\{a_k\}_{k\in \K}$ 
such that 
$ \| a \|_{ \ell^q (\K)} 
= 
\left( \sum_{ k \in \K } 
| a_k |^q \right)^{ 1/q } <\infty$ 
if $0< q < \infty$  
or 
$\| a \|_{ \ell^\infty (\K)} 
= \sup_{k \in \K} |a_k| < \infty$ 
if $q = \infty$.
For $0< q<\infty$, 
the space 
$\ell^{ q,\infty }(\K)$ is 
the set of all those complex 
sequences 
$a= \{ a_k \}_{ k \in \K }$ such that
\begin{equation*}
\| a \|_{\ell^{q,\infty} (\K) } 
=\sup_{t>0} 
\big\{ t\, 
\sharp \big( \{ k \in \K : | a_k | > t \} 
\big)^{1/q} \big\} < \infty,
\end{equation*}
where $\sharp$ denotes the cardinality 
of a set. 
Sometimes we write 
$\| a \|_{\ell^{q}} = 
\| a_k \|_{\ell^{q}_k }$ or 
$\| a \|_{\ell^{q,\infty}} = 
\| a_k \|_{\ell^{q,\infty}_k }$.  
If $\K=\Z^{n}$, 
we usually write $ \ell^q $ or $\ell^{ q, \infty }$ 
for 
$\ell^q (\Z^{n})$  
or $\ell^{q, \infty} (\Z^{n})$.

Let $X,Y,Z$ be function spaces.  
We denote the mixed norm by
\begin{equation*}
\label{normXYZ}
\| f (x,y,z) \|_{ X_x Y_y Z_z } 
= 
\bigg\| \big\| \| f (x,y,z) 
\|_{ X_x } \big\|_{ Y_y } \bigg\|_{ Z_z }.
\end{equation*} 
(Here pay special attention to the order 
of taking norms.)  
We shall use these mixed norms for  
$X, Y, Z$ being $L^p$ or $\ell^p$. 
The inequality
\begin{equation}\label{mixedtriangle}
\| f(\tau,\nu) \|_{\ell_{\tau}^p \ell_{\nu}^q }
\leq
\left\| \left\| f(\tau,\nu) \right\|_{\ell_{\nu}^q } \right\|_{\ell_{\tau}^{\min(p,q)}},
\quad 0 < p,q \leq \infty,
\end{equation}
will be often used,
and this can be proved by the inequality 
$\|\cdot \|_{\ell^p} \le \|\cdot \|_{\ell^q}$ for 
the case $p \ge q$
and the Minkowski inequality for the case $p \le q$.
It is worth mentioning that this inequality is incorrect for the $L^p$ case in general.

\subsection{Local Hardy spaces $h^p$ and the space $bmo$} 
\label{sechardybmo}

We recall the definition of the local Hardy spaces 
$h^p(\R^n)$, $0<p\leq\infty$, and the space $bmo(\R^n)$.

Let $\phi \in \calS(\R^n)$ be such that
$\int_{\R^n}\phi(x)\, dx \neq 0$. 
Then, the local Hardy space $h^p(\R^n)$ 
consists of
all $f \in \calS'(\R^n)$ such that 
$\|f\|_{h^p}=\|\sup_{0<t<1}|\phi_t*f|\|_{L^p}
<\infty$,  
where $\phi_t(x)=t^{-n}\phi(x/t)$.
It is known that $h^p(\R^n)$
does not depend on the choice of the function $\phi$,
and that $h^p(\R^n) = L^p(\R^n)$ for $1 < p \leq \infty$ 
and especially $h^1(\R^n) \hookrightarrow L^1(\R^n)$.

The space $bmo(\R^n)$ consists of
all locally integrable functions $f$ on $\R^n$ 
such that
\[
\|f\|_{bmo}
=\sup_{|R| \le 1}\frac{1}{|R|}
\int_{R}|f(x)-f_R|\, dx
+\sup_{|R|\geq1}\frac{1}{|R|}
\int_R |f(x)|\, dx
<\infty,
\]
where $f_R=|R|^{-1}\int_R f(x) \, dx$,
and $R$ ranges over the cubes in $\R^n$.

It is known that the dual space of $h^1(\R^n)$ is $bmo(\R^n)$.
See Goldberg \cite{goldberg 1979} for more details about 
$h^p$ and $bmo$. 

\subsection{Amalgam spaces} 
\label{secamalgam}

For $0 < p,q \leq \infty$, 
the amalgam space 
$(L^p,\ell^q)(\R^n)$ is 
defined to be the set of all 
those measurable functions $f$ on 
$\R^n$ such that 
\begin{equation*}
\| f \|_{ (L^p,\ell^q) (\R^n)} 
=
 \| f(x+\nu) \|_{L^p_x (Q) \ell^q_{\nu}(\Z^n) }
=\left\{ \sum_{\nu \in \Z^n}
\left( \int_{Q} \big| f(x+\nu) \big|^p \, dx 
\right)^{q/p} \right\}^{1/q} 
< \infty  
\end{equation*}
with usual modification when $p$ or $q$ is infinity.  
Obviously, $(L^p,\ell^p) = L^p$ 
and $(L^2, \ell^\infty) = L^2_{ul}$. 
For $1\leq p,q < \infty$, 
the duality $(L^p,\ell^q)^\ast =  (L^{p'},\ell^{q'}) $ holds.
If $p_1 \geq p_2$ and $q_1 \leq q_2$, 
then 
$(L^{p_1}, \ell^{q_1}) 
\hookrightarrow (L^{p_2},\ell^{q_2}) $.
In particular, $L^{q} \hookrightarrow (L^2,\ell^q)$
for $2 \le q \le \infty$,
$bmo \hookrightarrow (L^2,\ell^{\infty})$,
and $(L^2,\ell^q) \hookrightarrow h^q$ for $0 < q \leq 2$.

Here we prove the last embedding 
(the others are obvious).
In the definition of $h^q$, we choose 
$\phi \in \calS(\R^n)$ 
satisfying that $\supp \phi \subset [-1,1]^n$.
Then, by using $L^2(\R^n)\hookrightarrow L^2(Q)\hookrightarrow L^q(Q)$, we have 
\begin{align*}
\| f \|_{h^q} 
&=
\left\| \sup_{0<t<1} \left| \phi_{t} \ast (f \, \ichi_{\nu+3Q})(x+\nu)\right|
\right\|_{L^q_x(Q) \ell^q_{\nu}}
\\
&\lesssim
\left\| M (f \, \ichi_{\nu+3Q})(x) \right\|_{L^2_x(\R^n) \ell^q_{\nu}}
\lesssim
\left\| f \, \ichi_{\nu+3Q}\right\|_{L^2(\R^n) \ell^q_{\nu}}
\approx 
\left\| f \right\|_{(L^2, \ell^q)},
\end{align*}
where $M$ is the Hardy--Littlewood maximal operator
and $3Q=[-3/2,3/2)^n$.

See Fournier--Stewart \cite{fournier stewart 1985} 
and Holland \cite{holland 1975} for more 
properties of amalgam spaces.

We end this subsection with noting the following lemma.

\begin{lem}\label{lemequivalentamalgam}
Let $p,q \in (0,\infty]$. 
If $L> n/ \min (p,q)$ and if 
$g$ is a measuarable function on $\R^n$ such that 
\begin{equation}\label{assumptiong}
c \ichi_{Q} (x) \le |g(x)| \le c^{-1} \langle x \rangle ^{-L} 
\end{equation}
with some positive constant $c$, then 
\begin{equation}
\label{equivalentamalgam}
\| f \|_{ (L^p,\ell^q) (\R^{n})} 
\approx 
\left\| 
g(x-\nu ) f(x) \right\|_{L^p_x(\R^{n}) \ell^q_{\nu}(\Z^{n}) }.
\end{equation}
\end{lem}

\begin{proof}
The inequality $\lesssim$ in 
\eqref{equivalentamalgam} is obvious from the 
former inequality of \eqref{assumptiong}.  
To show the inequality $\gtrsim$ in 
\eqref{equivalentamalgam}, 
we use the formula \eqref{cubicdiscretization} 
to write 
the right hand side of \eqref{equivalentamalgam} 
as 
\begin{equation}
\label{aaa}
\left\| g( x-\nu )  
f(x) \right\|_{L^p_x(\R^{n}) \ell^q_{\nu} }
=
\left\| g (x+\tau -\nu )  
f(x+\tau) \right\|_{L^p_x(Q) \ell^p_\tau \ell^q_{\nu} }. 
\end{equation}
Using the latter inequality of \eqref{assumptiong}, 
changing variables, and using 
the inequality \eqref{mixedtriangle}, we see that 
the right hand side of \eqref{aaa} is bounded by
\begin{align*}
&\left\| 
\langle x+ \tau - \nu \rangle ^{-L} 
f(x+\tau) 
\right\|
_{L^p_x(Q) \ell^p_\tau \ell^q_{\nu} }
\approx 
\left\| \langle \tau  \rangle ^{-L} 
f(x+\tau+\nu) 
\right\|
_{L^p_x(Q) \ell^p_{\tau} \ell^q_{\nu}} 
\\
&
\le 
\left\| \langle \tau \rangle^{-L} 
\| f(x+\tau+\nu) \|_{L^p_x(Q) \ell^q_{\nu}} 
\right\|
_{\ell^{\min(p,q)}_\tau }
\approx
\| f \|_{(L^p, \ell^q)}, 
\end{align*}
where the last $\approx$ holds 
because $\min (p,q) L>{n}$. 
\end{proof}

\section{Class $\calB_N$}
\label{sectionWeight}

In this section, we give several properties of the class 
$\calB_N (\ZnN)$ introduced in 
Definition \ref{classB}. 
We also introduce the class $\calM (\R^d)$, 
which will be used in the next section. 
For the case $N=2$, most of the results in this section 
are already given in \cite[Section 3]{KMT-arXiv}. 
The argument in this section is a modification 
of \cite[loc.\ cit.]{KMT-arXiv} to cover all 
$N\ge 1$. 

\begin{prop}\label{propertiesB}
\begin{enumerate}
\setlength{\itemindent}{0pt} 
\setlength{\itemsep}{3pt} 
\item 
Every function in the class 
$\calB_N (\ZnN)$ 
is bounded. 
\item 
If $N=1$, then 
$\calB_N (\ZnN)=\calB_1 (\Z^n) = \ell ^{\infty} (\Z^n)$.  
\item 
Let $N\ge 2$. 
A nonnegative function $V=V(\nu_1, \dots, \nu_N)$ 
on $\ZnN$ belongs to 
$\calB_N (\ZnN)$ if and only 
if each of
\[
V(-\nu_1, \dots, -\nu_{j-1}, \nu_1+\dots+\nu_N, -\nu_{j+1}, \dots, -\nu_N),
\quad j=1,\dots,N,
\]
belongs to $\calB_N (\ZnN)$. 
\item 
If $N \ge 2$, then the class $\calB_N (\ZnN)$ is not 
rearrangement invariant, 
i.e., 
there exists a function $V$ on $\ZnN$ 
and a bijection $\Phi : \ZnN \to \ZnN$ 
such that  
$V \in \calB_N (\ZnN)$ but 
$V\circ \Phi \not\in \calB_N (\ZnN)$. 

\item 
 Let $d, d' \in \N$, 
$V\in \calB_N ((\Z^{d})^N)$, 
and 
$V'\in \calB_N ((\Z^{d'})^N)$.  
Then 
the function 
\[
W((\mu_1, \mu'_1), \dots, (\mu_N, \mu'_N))  
= V (\mu_1, \dots, \mu_N) 
V^{\prime} (\mu'_1, \dots, \mu'_N)
\]
for $\mu_{j} \in \Z^{d}$, 
$\mu'_{j} \in \Z^{d^{\prime}}$,
$j=1, \dots, N$,
belongs to 
$\calB_N ((\Z^{ d+d^{\prime} })^N)$. 
\end{enumerate}
\end{prop}

\begin{proof}
(1) If $V$ satisfies \eqref{BL222}, 
then applying it to the case where each of 
$A_0, A_1, \dots, A_N$ is a defining function  
of one point we easily find $V (\nu_1, \dots, \nu_N)\le c$. 

(2) The inclusion $\calB_1 (\Z^n) \subset 
\ell ^{\infty} (\Z^n)$ 
follows from (1). 
The opposite inclusion 
$\ell ^{\infty} (\Z^n) \subset 
\calB_1 (\Z^n)$ 
also holds because the 
Cauchy--Schwarz inequality 
implies 
\[
\sum_{\nu_1 \in \Z^n}
V (\nu_1) A_{0}(\nu_1) A_{1} (\nu_1)
\le 
\|V\|_{\ell^{\infty}}
\|A_0\|_{\ell^2} 
\|A_1\|_{\ell^2}.  
\]

(3) This can be easily proved by a simple change of 
variables.

(4) 
Let $N \ge 2$. 
The function 
$V (\nu_1, \dots, \nu_N) = (1+ |\nu_1| + \cdots + |\nu_{N-1}| )^{-(N-1)n/2-\epsilon}$ with $\epsilon >0$ 
belongs to $\calB_N (\ZnN)$; 
this can be seen from 
Proposition \ref{single-L2} (1) below 
since this $V$ is in $\ell^2((\Z^n)^{N-1})$.
On the other hand, for 
$0< \alpha < (N-1)n/2$, 
the function 
\[
W (\nu_1, \dots, \nu_N) = (1+ |\nu_1| + \cdots + |\nu_{N}| )^{-(N-1)n/2+\alpha},
\quad (\nu_1, \dots, \nu_N) \in \ZnN,
\] 
does not belong to $\calB_N (\ZnN)$. 
In fact, for $A_{j}(\mu) = \langle \mu \rangle^{-n/2-\alpha/(N+2)} \in \ell^{2}_{\mu}(\Z^n)$, 
$j=0, 1, \dots, N$,
the left hand side of \eqref{BL222} is bounded from below by
a constant (depending on $N$) times
\begin{align*}
&\sum_{\nu_1} \sum_{|\nu_2| \leq |\nu_1|/N} \cdots \sum_{|\nu_N| \leq |\nu_1|/N} 
\langle \nu_1 \rangle^{-(N-1)n/2+\alpha}
\left( \langle \nu_1 \rangle^{-n/2-\alpha/(N+2)} \right)^{N+1}
\\&\approx
\sum_{\nu_1} \langle \nu_1 \rangle^{-n+\alpha/(N+2)}
=\infty.
\end{align*}
Notice that 
both $V$ and $W$ take their values in 
the interval $(0, 1]$. 
For $\ell\in \N_0$, set 
\begin{align*}
&
E_{\ell} (V)= \{
(\nu_1, \dots, \nu_N) \in \ZnN : 
2^{-\ell-1}< V(\nu_1, \dots, \nu_N) \le 2^{-\ell}
\}, 
\\
&
E_{\ell} (W)= \{
(\nu_1, \dots, \nu_N) \in \ZnN : 
2^{-\ell-1}< W(\nu_1, \dots, \nu_N) \le 2^{-\ell}
\}. 
\end{align*}
Then the sets $\{E_{\ell}(V)\}_{\ell \in \N_0}$ 
and $\{E_{\ell}(W)\}_{\ell \in \N_0}$ 
are the partitions of $\ZnN$, 
each 
$E_{\ell}(V)$ is an infinite set, 
and 
$E_{\ell}(W)$ is a finite set.   
It is easy to construct a 
bijection $\Phi$ of $\ZnN$ onto itself 
such that 
\begin{equation*}
\Phi (E_{\ell}(W)) \subset E_{0}(V) \cup \dots \cup E_{\ell}(V) \
\;\; \text{for all}\;\; 
\ell \in \N_0. 
\end{equation*}
Then 
$W \le 2^{-\ell}$ 
and 
$V\circ \Phi > 2^{-\ell-1}$  
on each $E_{\ell}(W)$, we have 
$W < 2 V\circ \Phi $ on the whole $\ZnN$. 
Since $W \not\in \calB_N (\ZnN)$, we have  
$V\circ \Phi  \not\in \calB_N (\ZnN)$.

(5) 
This is easily checked from the definition of 
$\calB_{N}(\ZnN)$. 
\end{proof}

\begin{prop}\label{single-L2}
\begin{enumerate}
\setlength{\itemindent}{0pt} 
\setlength{\itemsep}{3pt} 


\item 
Let $N \geq 2$. 
Suppose 
a nonnegative function $V$ on 
$\ZnN$ is  
one of the following forms: 
\[
V(\nu_1, \dots, \nu_N)= 
V_{0}(\nu_1, \dots, \nu_{k-1}, \nu_{k+1}, \dots, \nu_N), 
\quad k=1, \dots, N.
\] 
Then $V \in \calB_N (\ZnN)$ 
if and only if 
there exists a constant $C\in (0, \infty)$ 
such that the inequality 
\begin{align}\label{V0L2iff}
\begin{split}
&
\sum_{\nu_1, \dots, \nu_{k-1}, \nu_{k+1}, \dots, \nu_N \in \Z^n} 
V_{0}(\nu_1, \dots, \nu_{k-1}, \nu_{k+1}, \dots, \nu_N)
\prod_{j\in\{1,\dots,N\}\setminus\{k\}} A_{j}(\nu_{j})
\\
&\leq C
\prod_{j\in\{1,\dots,N\}\setminus\{k\}} \| A_{j} \|_{\ell^2}
\end{split}
\end{align}
holds for all nonnegative 
functions $A_j$, $j\in\{1,\dots,N\}\setminus\{k\}$, on $\Z^n$. 
In particular, 
$V \in \calB_N (\ZnN)$ 
if and only if $V_0 \in \ell^2(\Z^n)$ when $N=2$, 
and $V \in \calB_N (\ZnN)$ 
if $V_0 \in \ell^2((\Z^n)^{N-1})$ when $N\geq3$.

\item 
Let $N \geq 2$.  
Then a nonzero constant function 
does not belong to 
$\calB_N ((\Z^n)^N)$. 

\item 
Let $N \geq 3$. 
Then 
any nonnegative function,
not identically zero, 
that is of the form 
$V (\nu_1, \dots, \nu_N) 
= 
V_{0}
(\nu_1, \dots, \nu_{N-2})$ 
does not belong to $\calB_N(\ZnN)$.
More generally, 
any nonnegative function,
not identically zero, 
that depends on only $N-2$ of the 
variables $\nu_1, \dots, \nu_N$ 
does not belong to 
$\calB_N(\ZnN)$.  
\end{enumerate}
\end{prop}

\begin{proof}
The assertion (1) for the case $N=2$ 
was already shown in \cite[Proposition 3.2]{KMT-arXiv}; 
the proof below will be an alternative one.
To prove (1), it suffices to consider the case $k=N$, 
i.e., the case where $V$ is of the form 
$V(\nu_1, \dots, \nu_N)= 
V_{0}(\nu_1, \dots, \nu_{N-1})$. 
To prove the ``if'' part, 
suppose \eqref{V0L2iff} with $k=N$ holds. 
Then, using the 
Cauchy--Schwarz inequality for the sum 
over $\nu_N$, we have 
\begin{align*}
&
\sum_{\nu_1, \dots, \nu_N \in \Z^n} 
V_0(\nu_1, \dots, \nu_{N-1}) 
A_0(\nu_1+\dots+ \nu_N) 
\prod_{j=1}^N A_{j}(\nu_{j})
\\
&\leq
\| A_{0} \|_{\ell^2 (\Z^n)}
\| A_{N} \|_{\ell^2 (\Z^n)}
\sum_{\nu_1, \dots, \nu_{N-1}}
V_0(\nu_1, \dots, \nu_{N-1})
\prod_{j=1}^{N-1} A_{j}(\nu_{j})
\lesssim
\prod_{j=0}^{N} \| A_{j} \|_{\ell^2 (\Z^n)}, 
\end{align*}
which implies $V \in \calB_{N}(\ZnN)$.
We next consider the ``only if'' part.
Assume that $V \in \calB_{N}(\ZnN)$, i.e., 
assume the inequality 
\begin{align*}
\sum_{\nu_1, \dots, \nu_N \in \Z^n} 
V_0(\nu_1, \dots, \nu_{N-1}) 
A_0(\nu_1+\dots+ \nu_N) 
\prod_{j=1}^N A_{j}(\nu_{j})
\lesssim
\prod_{j=0}^{N} \| A_{j}\|_{\ell^2 (\Z^n)} 
\end{align*}
holds. 
By duality, this is identical with
\begin{align}\label{V0L2(1)dual}
\begin{split}
&
\left\| \sum_{\nu_1, \dots, \nu_{N-1} \in \Z^n} 
V_0(\nu_1, \dots, \nu_{N-1}) 
A_0(\nu_1+\dots+ \nu_N) 
\prod_{j=1}^{N-1} A_{j}(\nu_{j}) \right\|_{\ell^2_{\nu_{N}} (\Z^n)}
\\
&\lesssim
\prod_{j=0}^{N-1} \| A_{j} \|_{\ell^2 (\Z^n)}.
\end{split}
\end{align}
We take a positive integer $M$ and set 
\begin{equation*}
A_{0} (\nu) = \left\{
\begin{array}{ll}
1, &\;\textrm{if}\; |\nu| \leq NM, \\
0, &\;\textrm{if}\; |\nu| > NM.
\end{array} \right.
\end{equation*}
Since $ V_0 $ and $A_{j}$, $j=1,\dots,N-1$, 
are nonnegative, the left hand side of \eqref{V0L2(1)dual} 
is bounded from below by
\begin{align*}
&
\left\| \sum_{|\nu_1|, \dots, |\nu_{N-1}| \leq M} 
V_0(\nu_1, \dots, \nu_{N-1}) 
A_0(\nu_1+\dots+ \nu_N) 
\prod_{j=1}^{N-1} A_{j}(\nu_{j}) \right\|_{\ell^2(\{ |\nu_{N}| \leq M\})}
\\
&=
\left\| \sum_{|\nu_1|, \dots, |\nu_{N-1}| \leq M} 
V_0(\nu_1, \dots, \nu_{N-1}) 
\prod_{j=1}^{N-1} A_{j}(\nu_{j}) \right\|_{\ell^2(\{ |\nu_{N}| \leq M\})}
\\
&\approx 
M^{n/2}
\sum_{|\nu_1|, \dots, |\nu_{N-1}| \leq M} 
V_0(\nu_1, \dots, \nu_{N-1}) 
\prod_{j=1}^{N-1} A_{j}(\nu_{j}). 
\end{align*}
Hence, \eqref{V0L2(1)dual} and the estimate 
$\| A_0 \|_{\ell^2} \approx M^{n/2}$ imply 
\begin{equation*}
\sum_{|\nu_1|, \dots, |\nu_{N-1}| \leq M} 
V_0(\nu_1, \dots, \nu_{N-1})
\prod_{j=1}^{N-1} A_{j}(\nu_{j})
\lesssim
\prod_{j=1}^{N-1} \| A_{j} \|_{\ell^2 (\Z^n)}.
\end{equation*}
Since the implicit constants above is independent of $M$,
by letting $M \to \infty$, 
we obtain \eqref{V0L2iff} with $k=N$. 
This completes the proof of the ``only if'' part.
The remaining claims concerning the case 
$V_0 \in \ell^2$
immediately follow from 
\eqref{V0L2iff} with the aid of 
duality or the Cauchy--Schwarz inequality; 
we omit the details.

For the assertion (2), 
the case $N=2$ is obvious from (1),
so that we shall investigate the cases $N\geq3$.
Toward a contradiction, 
we assume that the constant function $1$ belongs to $\calB_N(\ZnN)$,
that is, the inequality
\begin{equation*}
\sum_{\nu_1, \dots, \nu_N \in \Z^n} 
A_0(\nu_1+\dots+ \nu_N) 
\prod_{j=1}^N A_j(\nu_j)
\le c \prod_{j=0}^N
\|A_j\|_{\ell^2 (\Z^n) }
\end{equation*}
holds for all nonnegative 
functions $A_j$, $j=0,1,\dots,N$, on $\Z^n$. 
Test
\begin{equation}\label{testofconstnotinBN}
A_{j} (\nu_j) = \left\{
\begin{array}{ll}
1, &\;\textrm{if}\; \nu_{j}=0, \\
0, &\;\textrm{if}\; \nu_{j}\neq0
\end{array} \right.
\end{equation}
for $j=1,\dots,N-2$.
Then, we have
\begin{equation}\label{testedofconstnotinBN}
\sum_{\nu_{N-1}, \nu_N} 
A_0(\nu_{N-1}+ \nu_N) 
A_{N-1} (\nu_{N-1}) A_{N} (\nu_{N})
\le c 
\|A_0\|_{\ell^2}\|A_{N-1}\|_{\ell^2}\|A_N\|_{\ell^2},
\end{equation}
which means that the constant function $1$ belongs to $\calB_2 ((\Z^{n})^{2})$, 
which contradicts 
the result for the case $N=2$. 
Hence, we obtain the assertion (2) for $N\geq3$.

To prove (3), assume $V$ is of the form 
$V(\nu_1, \dots, \nu_N)= 
V_{0}(\nu_1, \dots, \nu_{N-2})
$ 
and 
assume $V\in \calB_N(\ZnN)$,
that is, the following inequality holds:
\begin{equation*}
\sum_{\nu_1, \dots, \nu_N \in \Z^n} 
V_{0}(\nu_1, \dots, \nu_{N-2}) 
A_0(\nu_1+\dots+ \nu_N) 
\prod_{j=1}^N A_j(\nu_j)
\le c \prod_{j=0}^N
\|A_j\|_{\ell^2 (\Z^n) }.
\end{equation*}
By testing $A_{j}$ given by 
\eqref{testofconstnotinBN} 
for $j=1, \dots, N- 2$, 
we have 
\begin{align*}
&
\sum_{\nu_{N-1}, \nu_N \in \Z^n} 
V_{0}(0, \dots, 0) 
A_0(\nu_{N-1}+ \nu_N) 
A_{N-1}(\nu_{N-1})
A_{N}(\nu_{N})
\\
&
\le c 
\|A_{0}\|_{\ell^2 (\Z^n) } 
\|A_{N-1}\|_{\ell^2 (\Z^n) } 
\|A_{N}\|_{\ell^2 (\Z^n) }.
\end{align*}
Then, by (2), we must have 
$V_{0}(0, \dots, 0)=0$. 
By translation, we also have 
$V_{0}(\nu_1, \dots, \nu_{N-2})=0$ 
for all $\nu_1, \dots, \nu_{N-2}$ 
and thus $V=0$. 
\end{proof}

\begin{prop}\label{weightLweak}
Let $N\ge 2$. 
Then all nonnegative functions in the class 
$\ell^{\frac{2N}{N-1}, \infty} (\ZnN)$  
belong to $\calB_N (\ZnN)$.  
\end{prop}

\begin{proof} 
By appropriately extending functions 
on $\Z^n$ and $\ZnN$ to functions on $\R^n$ 
and $(\R^n)^N$, 
it is sufficient to prove the inequality 
\begin{align}\label{BLNinfty2}
\begin{split}
&\int_{(\R^n)^N} 
V(x_1, \dots, x_{N})
A_0(x_1+\cdots+x_N)
\prod_{j=1}^N A_{j}(x_{j}) \, dx_1 \cdots dx_N
\\
&\lesssim 
\|V\|_{L^{\frac{2N}{N-1}, \infty}(\R^{nN})} 
\prod_{j=0}^N \| A_{j} \|_{L^{2}(\R^n)}
\end{split}
\end{align}
for nonnegative measurable functions $V, A_j$, $j=0, 1, \dots, N$,
on the corresponding Euclidean spaces.
We shall derive this inequality from a combination of real interpolation and
the inequality 
\begin{align}\label{BLqqi}
\begin{split}
&\int_{(\R^n)^N} 
V(x_1, \dots, x_{N})
A_0(x_1+\cdots+x_N)
\prod_{j=1}^N A_{j}(x_{j}) \, dx_1 \cdots dx_N
\\&\leq
\|V\|_{L^{q}(\R^{nN})} 
\prod_{j=0}^N \| A_{j} \|_{L^{q_{j}}(\R^n)}
\end{split}
\end{align}
for
\begin{align}
&\label{scaling}
\frac{N}{q}
+
\sum_{j=0}^N \frac{1}{q_{j}}
=N, 
\\
&\label{Youngconvolution}
0\le 
\frac{1}{q_{j}} \le 1 - 
\frac{1}{q} \le 1, 
\quad j=0, 1,\dots, N.    
\end{align}
Here we give a proof of 
\eqref{BLqqi}-\eqref{scaling}-\eqref{Youngconvolution}. 
By duality, it is sufficient to show 
\begin{equation}\label{HolderYoung}
\left\| A_0(x_1+\cdots+x_N) 
\prod_{j=1}^N A_{j}(x_{j}) 
\right\|_{L^{q'}_{x_1, \dots, x_N} (\R^{nN})} 
\leq
\prod_{j=0}^N \| A_{j} \|_{L^{q_{j}} (\R^n)}. 
\end{equation}
In the case $q'=\infty$, \eqref{Youngconvolution} implies 
$q_j=\infty$, $j=0,1,\dots,N$, and \eqref{HolderYoung} is obvious. 
We assume $q' < \infty$. 
Let $\alpha_j \in [1, \infty]$, $j=0,1,\dots,N$,
satisfy $\sum_{j=0}^{N} 1/\alpha_{j} = N$. 
Then writing $\widetilde{A}(x)= A (-x)$,  
using H\"older's inequality with 
$1/ \alpha = 1- 1/ \alpha_N$, 
and using 
$(N-1)$-times Young's inequalities for convolution, 
we have 
\begin{align*}
&
\left\| A_0(x_1+\cdots+x_N) \prod_{j=1}^N A_j(x_j) \right\|_{L^{q'}_{x_1, \dots, x_N}}^{q'}
\\&=
\int 
(\widetilde{A_0}^{q'}\ast A_1^{q'} \ast \cdots \ast A_{N-1}^{q'})(-x_N)
A_N(x_N)^{q'}
\, dx_N 
\\
&\le 
\|\widetilde{A_0}^{q'}\ast A_1^{q'} \ast \cdots \ast A_{N-1}^{q'}\|_{ L^\alpha }
\|A_N^{q'}\|_{ L^{\alpha_N} }
\\
&\le 
\|\widetilde{A_0}^{q'}\|_{L^{\alpha_0}} 
\|A_1^{q'}\|_{L^{\alpha_1}} \cdots
\|A_{N-1}^{q'}\|_{L^{\alpha_{N-1}}} 
\|A_N^{q'}\|_{ L^{\alpha_N} }
\\
&=
\left( \prod_{j=0}^N \|A_j\|_{L^{q' \alpha_j}} \right)^{q'}.
\end{align*}
Here, observe that
the conditions set on 
$\alpha_j$ and $\alpha$ 
ensure that we can use 
Young's inequalities $(N-1)$-times.
By choosing $\alpha_j$ such that 
$q' \alpha_j=q_j$, $j=0,1,\dots,N$, 
we obtain \eqref{HolderYoung}.

Now, from \eqref{BLqqi}, it follows by duality 
that 
the multilinear map 
\[
T(A_0,A_1,\dots,A_{N})(x_1, \dots, x_N)
=A_0(x_1+\cdots+x_N) A_1(x_1)\dots A_N(x_N)
\]
satisfies the estimate 
\begin{equation*}
\|T(A_0,A_1,\dots,A_{N})\|_{L^{q'} (\R^{nN})}
\leq 
\prod_{j=0}^N \|A_j\|_{L^{q_j} (\R^n) }
\end{equation*}
for all $q$ and $(q_{j})$ 
satisfying \eqref{scaling} 
and \eqref{Youngconvolution}. 
Hence, by the real interpolation for multilinear 
operators (see Janson \cite{Janson}),  
it follows that 
if $q$ and $(q_j)$ satisfy \eqref{scaling} and also 
satisfy 
the strict inequalities 
\begin{equation}\label{strict}
0<\frac{1}{q_j} < 1 - 
\frac{1}{q}<1, 
\quad j=0, 1,\dots, N,
\end{equation}
then the Lorentz norm estimate 
\begin{equation*}
\|T(A_0,A_1,\dots,A_{N})\|_{L^{q',r'} (\R^{nN})}
\lesssim 
\prod_{j=0}^N \|A_j\|_{L^{q_j,r_j} (\R^n) }
\end{equation*}
holds for all $r$ and $(r_j)$ such that 
\begin{equation}\label{Holder}
r, r_j\in [1, \infty], \quad 
j=0, 1,\dots, N,
\quad \text{and}
\quad 
\frac{1}{r} +
\sum_{j=0}^N \frac{1}{r_j}
=1. 
\end{equation}
By duality again, 
this implies that 
the inequality 
\begin{equation}\label{BLLorentz}
\begin{split}
&\int_{(\R^n)^N} 
V(x_1, \dots, x_{N})
A_0(x_1+\cdots+x_N)
\prod_{j=1}^N A_{j}(x_{j}) \, dx_1 \cdots dx_N
\\
&\lesssim 
\|V\|_{
L^{q,r} (\R^{nN}) }
\prod_{j=0}^N \|A_{j}\|_{L^{q_{j},r_{j}} (\R^n) }
\end{split}
\end{equation}
holds for all $q$, $(q_{j})$, $r$, and $(r_{j})$ satisfying 
\eqref{scaling}, \eqref{strict}, and \eqref{Holder}. 
In particular, by taking 
$q=\frac{2N}{N-1}$, 
$q_0=q_1=\cdots=q_N=2$, 
$r=r_0=r_1=\cdots=r_{N-2}=\infty$, 
and 
$r_{N-1}=r_N=2$, 
we obtain 
\begin{align*}
&\int_{(\R^n)^N} 
V(x_1, \dots, x_{N})
A_0(x_1+\cdots+x_N)
\prod_{j=1}^N A_{j}(x_{j}) \, dx_1 \cdots dx_N
\\
&
\lesssim 
\|V\|_{
L^{\frac{2N}{N-1},\infty}(\R^{nN}) }
\left(\prod_{j=0}^{N-2} \|A_{j}\|_{L^{2,\infty} (\R^n) }\right)
\|A_{N-1}\|_{L^{2} (\R^n)}
\|A_{N}\|_{L^{2} (\R^n)},  
\end{align*}
which {\it a fortiori\/} implies \eqref{BLNinfty2}. 
\end{proof}

\begin{prop}\label{productLweak}
Let $N \ge 2$, 
$2<p_{j} < \infty$, 
$\sum_{j=1}^N 1/p_{j} = (N-1)/2$, 
and let 
$f_{j} \in \ell^{p_{j}, \infty}(\Z^n)$
be nonnegative sequences
for $j=1,\dots,N$. 
Then the function 
$\prod_{j=1}^N f_{j} (\nu_{j})$
belongs to $\calB_N (\ZnN)$.  
\end{prop}

\begin{proof} 
In \eqref{BLqqi}-\eqref{scaling}-\eqref{Youngconvolution} 
given in Proof of Proposition \ref{weightLweak}, 
consider the special case that $V(x_1, \dots, x_N)=1$ and $q=\infty$. 
Then replacing $A_{j}(x_j)$ by $f_j (x_j) A_j (x_j)$ 
and $1/q_j$ by $1/p_j + 1/q_j$ for $j=1, \dots, N$, 
we obtain the inequality
\begin{align}
\label{productweaklemma}
\begin{split}
&\int_{(\R^n)^N} 
A_0(x_1+\cdots+x_N) 
\prod_{j=1}^N f_{j}(x_{j}) A_{j}(x_{j})
\, dx_1 \cdots dx_N
\\&\leq
\|A_0\|_{L^{q_0}(\R^n)}
\prod_{j=1}^N \|f_{j}\|_{L^{p_{j}}(\R^n)} 
\|A_{j}\|_{L^{q_{j}}(\R^n)}
\end{split}
\end{align}
for 
\begin{align}
&
\frac{1}{q_0} + \sum_{j=1}^N \left(\frac{1}{p_{j}}+\frac{1}{q_{j}}\right) 
=N, 
\label{HolderYoung1}
\\
&
0\le 1/q_0,1/p_{j}, 1/q_{j} \le 1, 
\label{HolderYoung2}
\\
&
0\le {1}/{p_{j}}+{1}/{q_{j}}\le 1,
\quad j=1,\dots,N.
\label{HolderYoung3}
\end{align}
Hence, by the same argument of interpolation as in 
Proof of Proposition 
\ref{weightLweak}, 
we see that 
\eqref{productweaklemma} holds 
with the Lebesgue norms replaced by 
appropriate Lorentz norms 
if the equality \eqref{HolderYoung1} holds 
and if all the inequalities 
\eqref{HolderYoung2} and \eqref{HolderYoung3}
hold with strict inequalities. 
Thus, in particular, 
for $q_0=q_1=\cdots=q_N=2$ and for $p_1, \dots, p_N$ satisfying 
$0<1/p_{j}< 1/2$ and $1/p_1+\cdots+1/p_N=(N-1)/2$, 
we have 
\begin{align*}
&
\int_{(\R^n)^N} 
\left(\prod_{j=1}^N f_j(x_j)\right)
A_0(x_1+\cdots+x_N) 
\prod_{j=1}^N A_{j}(x_{j})
\, dx_1 \cdots dx_N
\\
&
\lesssim 
\left(\prod_{j=1}^{N} 
\|f_{j}\|_{ L^{p_{j},\infty} (\R^n) }\right)
\left(\prod_{j=0}^{N-2} 
\|A_{j}\|_{L^{2,\infty} (\R^n)}\right)
\|A_{N-1}\|_{L^{2} (\R^n)}
\|A_{N}\|_{L^{2} (\R^n)}, 
\end{align*}
which {\it a fortiori\/} implies the 
conclusion of Proposition \ref{productLweak}.  
\end{proof}

Here, we prove Example \ref{example-of-V}.

\begin{proof}[Proof of Example \ref{example-of-V}]
If $N=1$, then the function \eqref{V-n/2} obviously 
belongs to $\calB _1 (\Z^n)= \ell^{\infty} (\Z ^n)$. 
If $N\ge 2$, then 
\eqref{V-n/2} 
belongs to $\calB_N (\ZnN)$ 
by Proposition \ref{weightLweak} 
and \eqref{Vproduct2}
belongs to $\calB_N (\ZnN)$ 
by Proposition \ref{productLweak} 
with 
$p_j = n/a_j$. 
\end{proof}

We introduce the following.

\begin{defn}\label{defModerate}
Let $d \in \N$. 
We say that a continuous function 
$F: \R^d \to (0, \infty)$   
is of {\it moderate class\/} if  
there exists constants $C=C_{F}>0$ 
and $M=M_{F}>0$ such that 
\begin{equation}\label{moderateabove}
F(\xi + \eta ) 
\le C F(\xi)
\langle \eta \rangle ^{M}
\; \; \text{for all} \;\; \xi, \eta \in \R^d.  
\end{equation}
We denote by 
$\calM (\R^d)$ the set of 
all functions on $\R^d$ of moderate class. 
\end{defn}

This class was defined 
in \cite[Definition 3.7]{KMT-arXiv} 
in a slightly different way. 
Notice that the inequality 
\eqref{moderateabove} implies 
\begin{equation}\label{moderateboth}
C^{-1}
F(\xi) \langle \eta \rangle ^{-M}
\le 
F(\xi - \eta) 
\le 
C 
F(\xi)
\langle \eta \rangle ^{M}. 
\end{equation}
Hence, if $p\in (0, \infty)$ and 
if $L=L_{p,M}>0$ is sufficiently large, 
we have 
\begin{equation}\label{moderateintegral}
\bigg(
\int_{\R^d} 
F(\xi - \eta)^p 
\langle \eta \rangle ^{-L}\, 
d \eta \bigg)^{1/p}
\approx 
F (\xi ). 
\end{equation}
Conversely if \eqref{moderateintegral} holds, 
then \eqref{moderateabove} holds with 
$M= L/p$. 
Thus the condition 
\eqref{moderateintegral} 
also characterizes $F\in \calM (\R^d)$. 
The case $p=2$ of \eqref{moderateintegral} 
was the condition used in 
 \cite[Definition 3.7]{KMT-arXiv}.

We give a general result 
concerning the classes $\calB$ and $\calM$. 

\begin{prop} \label{Vast}
For any $V\in \calB_N (\ZnN)$, 
there exists a function 
$V^{\ast} \in \calM (\R^{Nn})$ 
such that 
$V(\nu_1, \dots, \nu_N) \le V^{\ast} (\nu_1, \dots, \nu_N)$ 
for all $(\nu_1, \dots, \nu_N) \in \ZnN$ and 
the restriction of $V^{\ast}$ to $\ZnN $ 
belongs to $\calB_N (\ZnN)$. 
\end{prop}

\begin{proof} 
This proposition for the case $N=2$ was 
proved in \cite[Proposition 3.8]{KMT-arXiv}. 
Here we give a slightly simpler 
proof for general $N$.  
Suppose 
$V\in \calB_N ( \ZnN )$ 
and suppose the inequality 
\eqref{BL222} holds. 
We may assume $V$ is not identically equal to $0$. 
By translation of variables, 
we see that the inequality 
\begin{align}\label{BL222translation}
\begin{split}
\sum_{\nu_1, \dots, \nu_N \in \Z^n} 
V(\nu_1 - \mu_1, \dots, \nu_N - \mu_N) 
A_0(\nu_1+\cdots+ \nu_N) 
\prod_{j=1}^N A_{j}(\nu_{j})
\le c 
\prod_{j=0}^N \|A_{j}\|_{\ell^2 (\Z^n) } 
\end{split}
\end{align} 
holds for all $(\mu_1, \dots, \mu_N) \in \ZnN $ 
with the same constant $c$ as in \eqref{BL222}. 
Take a number $M> Nn$. 
Multiplying \eqref{BL222translation} 
by $\langle  (\mu_1, \dots, \mu_N)  \rangle ^{-M}$ 
and taking sum over 
$ \mu_1, \dots, \mu_N\in \Z^n$, 
we see that the function 
\begin{align*}
G(\nu_1, \dots, \nu_N) 
&
= \sum_{\mu_1, \dots, \mu_N \in \Z^n} 
V(\nu_1 - \mu_1, \dots, \nu_N - \mu_N) 
\langle  (\mu_1, \dots, \mu_N)  \rangle ^{-M} 
\\
&= \sum_{\mu_1, \dots, \mu_N \in \Z^n} 
V(\mu_1, \dots, \mu_N) 
\langle  ( \nu_1-\mu_1, \dots, \nu_N - \mu_N)  
\rangle ^{-M} 
\end{align*} 
also belongs to the class $\calB_N ((\Z^n)^N)$. 
Obviously 
$G(\nu_1, \dots, \nu_N) 
\ge V(\nu_1, \dots, \nu_N)$. 
We define $V^{\ast}$ 
on $\R^{nN}$ by 
\[
V^{\ast} (\xi_1, \dots, \xi_N)
=
\sum_{\mu_1, \dots, \mu_N \in \Z^n} 
V(\mu_1, \dots, \mu_N) 
\langle  ( \xi_1-\mu_1, \dots, \xi_N - \mu_N)  
\rangle ^{-M},  
\]
which is an extension of $G$ to $\R^{Nn}$. 
Then $V^{\ast}$ is 
in the moderate class $\calM (\R^{Nn})$ and 
has the desired properties. 
\end{proof}

\section{Lemmas}\label{sectionLemma}

In this section, we denote by $S$ the operator 
\begin{equation*}
S (f) (x) 
= \int_{\R^n} \frac{ f(y) }{ \langle x-y \rangle^{L} } dy
\end{equation*}
for sufficiently large $L>n$.

We will give several properties of the operator $S$.

\begin{lem}\label{propertiesofS} 
The following (1)--(4) 
hold for all nonnegative measurable 
functions $f, g$ on $\R^n$. 
\begin{enumerate}
\setlength{\itemindent}{0pt} 
\setlength{\itemsep}{3pt} 
\item 
$S(f \ast g)(x) = \big( S(f) \ast g \big)(x) =  
\big( f \ast S(g) \big)(x)$.
\item 
$S\left( S(f) \right)(x) \approx S(f) (x) $. 
\item 
$S(f)(x) \approx S(f)(y)$
for $x,y\in \R^n$ such that $|x-y|\lesssim1$.
\item 
$S\left (\ichi_{Q}\ast f\right)(x) 
\approx S(f)(x)$. 
\item 
Let $\varphi$ be a function in $\calS(\R^n)$
with compact support.
Then, $| \varphi(D-\nu) f(x) |^{2}
\lesssim S( | \varphi(D-\nu) f |^{2} )(x)$
for any $f \in \calS(\R^{n})$, $\nu\in\Z^n$, and $x\in\R^n$.
\end{enumerate}
\end{lem}

\begin{proof}
The assertion (1) follows from the associative and 
commutative laws for the convolution.
The assertion (2)  
follows from 
the relation 
$\langle \cdot \rangle^{-L} \ast 
\langle \cdot \rangle^{-L} \approx \langle \cdot \rangle^{-L}$, which holds for $L>n$. 
The assertion (3) is obvious from the fact 
$\langle x-z \rangle \approx \langle y-z \rangle$ 
for $z\in \R^n$ and $|x-y|\lesssim1$.
The assertion (4) follows 
from 
$\langle \cdot \rangle ^{-L} \ast \ichi_{Q} 
\approx \langle \cdot \rangle ^{-L} $. 
Finally, we shall prove the assertion (5).
Taking $\widetilde \varphi \in \calS(\R^n)$ satisfying
$\widetilde \varphi=1$ on $\mathrm{supp}\, \varphi$,
we can write
\begin{equation*}
\varphi(D-\nu) f(x) 
=
\widetilde \varphi(D-\nu)\varphi(D-\nu) f(x).
\end{equation*}
Then, the kernel representation,
\begin{equation*}
\varphi(D-\nu) f(x) 
=
\int_{\R^n} e^{i(x-y)\cdot \nu} 
{\calF^{-1} \widetilde \varphi}(x-y) \, \varphi(D-\nu) f(y) dy, 
\end{equation*}
provides by the use of the Cauchy--Schwarz inequality that
\begin{align*}
\left| \varphi(D-\nu) f(x) \right|^{2}
&\lesssim
\int_{\R^n} \left| {\calF^{-1}\widetilde \varphi}(x-y) \right| \left| \varphi(D-\nu) f(y) \right|^{2} dy
\\
&\lesssim 
S( | \varphi(D-\nu) f |^{2} )(x),
\end{align*}
which is the desired result.
\end{proof}

\begin{lem}[{\cite[Lemma 3.2]{miyachi tomita 2018 AIFG}}]
\label{estSRx}
Let $\varphi \in \calS (\R^n)$. 
Then we have
\begin{equation*}
\left\| \varphi ( D - \nu ) h (x) \right\|_{\ell^2_{\nu}}
\lesssim
\left\{ S \left( |h|^2 \right) (x) \right\}^{1/2}
\end{equation*}
for any $x\in\R^n$.
\end{lem}

\begin{lem}
\label{SellequivL}
Let $p \in (0, \infty]$.
Then the following hold.
\begin{enumerate}
\setlength{\itemindent}{0pt} 
\setlength{\itemsep}{5pt} 
\item 
$\| S(f)(x) \|_{L^p_x(\R^n)}
\approx \| S(f)(\nu) \|_{\ell^p_\nu (\Z^n)}$ 
for all nonnegative $f$. 
\item 
If the number $L$ in the definition of 
$S$ satisfies $\min (1, p/2) L >n$, then 
$\left\| 
\left\{ S \left( |f|^2 \right) (x) \right\}^{1/2}
\right\|_{L^p_x(\R^n)}
\approx
\| f \|_{(L^2,\ell^p)(\R^n)}$.
\end{enumerate}
\end{lem}

\begin{proof}
The first assertion immediately
follows from Proposition \ref{propertiesofS} (3).

For the second assertion, we have 
\begin{align*}
&
\left\| \left\{ S 
\left( |f|^2 \right)(x) \right\}^{1/2} \right\|_{L^{p}_x}
\approx 
\left\| 
S \left( |f|^2 \right)(\nu) \right\|_{\ell^{p/2}_\nu} ^{1/2}
\\
&=
\left\| 
\langle \nu - y \rangle ^{-L} |f(y)|^2 
\right\|
_{ L^1_{y}(\R^n) \ell^{p/2}_\nu (\Z^n)}
^{1/2}
\approx
\left\| |f|^2 
\right\|_{ (L^1, \ell^{p/2}) } ^{1/2}  
= 
\| f \|_{(L^2,\ell^p)}, 
\end{align*}
where the former $\approx$ follows from the first assertion 
and the 
latter $\approx$ follows from 
Lemma \ref{lemequivalentamalgam}. 
\end{proof}

We end this section by mentioning a lemma that 
can be found in Sugimoto \cite[Lemma 2.2.1]{Sugimoto}. 
We give a proof for the reader's convenience.

\begin{lem}
\label{unifdecom}
There exist functions $\kappa \in \calS(\R^n)$ and $\chi \in \calS(\R^n)$ 
such that
$\supp \kappa \subset [-1,1]^n$, $\supp \widehat \chi \subset B_1$, 
$| \chi | \geq c > 0$ on $[-1,1]^n$
and
\begin{equation*}
\sum_{\nu\in\Z^n} \kappa (\xi - \nu) \chi (\xi - \nu ) = 1, \quad \xi \in \R^n.
\end{equation*}
\end{lem}

\begin{proof}
The existence of a function $\chi \in \calS (\R^n)$ satisfying
$\supp \widehat \chi \subset B_1$ and
$| \chi | \geq 1$ on $[-1,1]^n$
is well-known.
Furthermore, we also know well that
there exist a partition of unity such that $\varphi\in\calS(\R^n)$,
\begin{equation*}
\supp \varphi \subset [-1,1]^{n},
\quad \textrm{and} \quad
\sum_{\nu\in\Z^n} \varphi (\xi - \nu) = 1, \quad \xi \in \R^n.
\end{equation*}
Therefore, we see that
\begin{equation*}
1 
= \sum_{\nu\in\Z^n} \varphi (\xi - \nu) 
= \sum_{\nu\in\Z^n} \frac{ \varphi (\xi - \nu) }{ \chi (\xi - \nu) } \chi (\xi - \nu ),
\end{equation*}
and that $\varphi/\chi \in \calS(\R^n)$ with $\supp \varphi/\chi \subset [-1,1]^n$.
Hence, choosing $\kappa=\varphi/\chi$, we complete the proof.
\end{proof}

\section{Main results}\label{sectionMain}
\subsection{Key proposition for Theorem \ref{main-thm-1}}
In this subsection, we prove 
Proposition \ref{main-prop} below, 
which plays a crucial role in our argument. 
In the succeeding subsections, 
we shall prove our main results, 
Theorems \ref{main-thm-1} and \ref{main-thm-2},  
by utilizing this proposition. 
The basic idea contained in 
these argument 
goes back to 
Boulkhemair 
\cite{boulkhemair 1995}.

\begin{prop}
\label{main-prop}
Let $W \in \calM (\R^{Nn})$ 
and suppose 
the restriction of $W$ to $\ZnN$ belongs 
to the class $\calB_N (\ZnN)$. 
For $j=1,\dots, N$, let 
$R_{0}, R_{j} \in [1, \infty)$, 
$q_{j},r \in (0, \infty]$, 
and $\sum_{j=1}^N 1/q_j = 1/r$.
Suppose $\sigma$ is 
a bounded continuous function on 
$(\R^n)^{N+1}$ such that  
$\supp \calF \sigma \subset 
B_{R_0} \times B_{R_1} \times \cdots \times B_{R_N}$. 
Then 
\begin{equation}
\label{conclusionProp}
\begin{split}
&
\left\| T_{ \sigma } \right\|_{(L^2,\ell^{q_1}) \times \dots \times (L^2,\ell^{q_N}) \to (L^2, \ell^{r})}
\\
&\lesssim
R_{0}^{n/2}
\bigg(\prod_{j=1}^N R_{j}^{ n/\min(2,q_{j}) }\bigg) 
\big\|
W(\xi_1, \dots, \xi_N)^{-1} 
\sigma (x,\xi_1,\dots,\xi_N) \big\|_{ L^{2}_{ul} ((\R^{n})^{N+1})}. 
\end{split}
\end{equation}
\end{prop}

\begin{proof}
In this proof, 
we will write $\boldsymbol{\xi}=(\xi_1,\dots,\xi_N) \in (\R^n)^N$
and $\boldsymbol{\nu}=(\nu_1,\dots,\nu_N) \in (\Z^n)^N$
for the sake of simplicity.
By this notation,  
$\sigma (x,\boldsymbol{\xi}) = \sigma (x,\xi_1,\dots,\xi_N)$,
$W(\boldsymbol{\xi}) = W(\xi_1, \dots, \xi_N)$,
$\boldsymbol{\xi} \pm \boldsymbol{\nu} = (\xi_1\pm \nu_1,\dots,\xi_N \pm \nu_N) $,
and $d\boldsymbol{\xi} = d\xi_1 \cdots d\xi_N$.

Take a $\theta \in \calS(\R^n)$ such that $|\theta| \geq c >0$ 
on $Q=[-1/2,1/2)^n$ and 
$\supp \widehat \theta \subset B_1$.
We realize from Lemma \ref{lemequivalentamalgam} that
\begin{equation*}
\left\| T_{ \sigma }( f_1, \dots, f_N ) \right\|_{(L^2,\ell^r)}
\approx
\left\| \theta (x-\mu) T_{ \sigma }( f_1, \dots, f_N )(x) \right\|_{L^2_x(\R^n) \ell^r_{\mu}(\Z^n) },
\end{equation*}
and by duality
\begin{equation}
\label{saishodual}
\left\| T_{ \sigma }( f_1, \dots, f_N ) \right\|_{(L^2,\ell^r)}
\approx
\left\| \sup_{\|g\|_{L^2}=1} \left| \int_{\R^n} \theta (x-\mu) T_{ \sigma }( f_1, \dots, f_N )(x) g(x) \, dx \right|\right\|_{ \ell^r_{\mu} }.
\end{equation}
Hence, in what follows we consider
\begin{equation*}
I = \int_{\R^n} \theta (x-\mu) T_{ \sigma }( f_1, \dots, f_N )(x) g(x) \, dx , 
\end{equation*}
for  $\mu\in\Z^n$ and $g\in L^2(\R^n)$.

Before estimating the integral $I$, we rewrite it. 
Firstly, by using Lemma \ref{unifdecom}, 
we decompose the symbol $\sigma$ as
\begin{align*}
\sigma ( x, \boldsymbol{\xi})
&=
\sum_{\boldsymbol{\nu}\in(\Z^n)^N} 
\sigma ( x, \boldsymbol{\xi} ) 
\prod_{j=1}^N \kappa (\xi_{j} - \nu_{j}) \chi (\xi_{j} - \nu_{j} )
\\&=
\sum_{\boldsymbol{\nu}\in(\Z^n)^N} 
\sigma_{\boldsymbol{\nu}} ( x, \boldsymbol{\xi} ) 
\prod_{j=1}^N \kappa (\xi_{j} - \nu_{j})
\end{align*}
with
\[
\sigma_{\boldsymbol{\nu}} ( x, \boldsymbol{\xi} ) 
=
\sigma ( x, \boldsymbol{\xi} ) 
\prod_{j=1}^N \chi (\xi_{j} - \nu_{j} ).
\]
Then, by denoting
the Fourier multiplier operator 
$\kappa (D - \nu_{j})$
by $\square_{\nu_{j}}$, $j=1,\dots,N$,  
the integral $I$ is written as 
\begin{align}
\label{decomposedI}
\begin{split}
I&=
\sum_{\boldsymbol{\nu}\in(\Z^n)^N} 
\int_{\R^n} \theta (x-\mu)
T_{ \sigma_{\boldsymbol{\nu}} }( \kappa(D-\nu_1)f_1, \dots, \kappa(D-\nu_N) f_N )(x) 
g(x) 
\, dx 
\\&=
\sum_{\boldsymbol{\nu}\in(\Z^n)^N} 
\int_{\R^n} \theta (x-\mu)
T_{ \sigma_{\boldsymbol{\nu}} }( \square_{\nu_1}f_1, \dots, \square_{\nu_N}f_N )(x) 
g(x) 
\, dx .
\end{split}
\end{align}
(The idea of decomposing 
pseudo-differential operators 
by the use of  
$\kappa$ and $\chi$ goes back to 
Sugimoto \cite{Sugimoto}.)

Secondly, in the integral of \eqref{decomposedI}, 
we transfer the information of the Fourier transform of 
$\theta (\cdot-\mu) T_{ \sigma_{\boldsymbol{\nu}} }( \square_{\nu_1}f_1, \dots, \square_{\nu_N}f_N )$ to $g$.
Taking the Fourier transform to 
$T_{ \sigma_{\boldsymbol{\nu}} } (\cdots)$,
we have
\begin{align*}
&(2\pi)^{Nn}
\calF \left[ T_{ \sigma_{\boldsymbol{\nu}} }( \square_{\nu_1}f_1, \dots, \square_{\nu_N}f_N ) \right](\zeta)
\\&=
\int_{ (\R^{n})^N }
\big( \calF_0 \sigma_{\boldsymbol{\nu}} \big) \big( \zeta - (\xi_1 +\cdots+ \xi_N) , \boldsymbol{\xi} \big) 
\, \prod_{j=1}^N \kappa (\xi_{j}-\nu_{j}) \widehat{f_{j}}(\xi_{j})
\, d\boldsymbol{\xi} .
\end{align*}
Combining this with the facts
$\supp \calF_{0} \sigma_{\boldsymbol{\nu}} (\cdot, \boldsymbol{\xi} ) \subset B_{R_0}$
and $\supp \kappa(\cdot-\nu_{j}) \subset \nu_{j}+[-1,1]^n$, $j=1,\dots,N$, 
we have
\begin{equation*}
\supp 
\calF \left[ T_{ \sigma_{\boldsymbol{\nu}} }( \square_{\nu_1}f_1, \dots, \square_{\nu_N}f_N ) \right]
\subset 
\big\{ \zeta \in \R^n : | \zeta-(\nu_1 +\dots+ \nu_N) | \lesssim R_0 \big\}.
\end{equation*}
Since $\supp \widehat \theta \subset B_1$, we see that 
\begin{equation*}
\supp 
\calF \left[ \theta (\cdot-\mu)T_{ \sigma_{\boldsymbol{\nu}} }( \square_{\nu_1}f_1, \dots, \square_{\nu_N}f_N ) \right]
\subset 
\big\{ \zeta \in \R^n : | \zeta-(\nu_1 +\dots+ \nu_N) | \lesssim R_0 \big\}
\end{equation*}
for any $\mu\in\Z^n$.
We take a function $\varphi \in \calS(\R^n)$ such that $\varphi = 1$ 
on $\{ \zeta\in\R^n: |\zeta|\lesssim 1\}$. 
Then the integral $I$ given in \eqref{decomposedI} 
can be further rewritten as
\begin{align}
\label{decomposedFouriertransI}
\begin{split}
I=
\sum_{\boldsymbol{\nu}\in(\Z^n)^N} \int_{\R^n} 
&\theta (x-\mu) T_{ \sigma_{\boldsymbol{\nu}} }
( \square_{\nu_1}f_1, \dots, \square_{\nu_N}f_N )(x) 
\\
&\times \varphi 
\left( \frac{D+ \nu_1 +\dots+ \nu_N}{R_0} \right) 
g (x) \, dx.
\end{split}
\end{align}

Now, we shall actually estimate the integral $I$ given 
in \eqref{decomposedFouriertransI}.
Observing that
\[
\left(
\calF_{1,\dots,N} \sigma_{\boldsymbol{\nu}} \right)
(x, \eta_1, \dots, \eta_N)
= 
\left(
\calF_{1,\dots,N} \sigma \right)
(x, \eta_1, \dots, \eta_N) 
\ast_{\boldsymbol{\eta}} 
\left( \prod_{j=1}^N e^{-i \nu_{j} \cdot \eta_{j}} \widehat{\chi}(\eta_{j}) \right)
\]
and then recalling that
$\supp \big(\calF_{1,\dots,N} \sigma \big)
(x, \cdot, \dots, \cdot)
\subset B_{R_1} \times \cdots \times B_{R_N}$
and $\supp \widehat{\chi} \subset B_1$,
we see that
\[
\supp \big(\calF_{1,\dots,N} \sigma_{\boldsymbol{\nu}} 
\big) (x, \cdot, \dots, \cdot)
\subset  
B_{2R_1} \times \cdots \times B_{2R_N}.
\]
This yields that
\begin{align*}
&
(2\pi)^{Nn} \, T_{ \sigma_{\boldsymbol{\nu}} }
( \square_{\nu_1}f_1, \dots, \square_{\nu_N}f_N )(x) 
\\&=
\int_{(\R^{n})^{N}}
\big( \calF_{1,\dots,N} \sigma_{\boldsymbol{\nu}} \big) (x, y_1-x, \dots, y_N-x)
\, \prod_{j=1}^N \square_{\nu_{j}} f_{j} (y_{j}) 
\, d\boldsymbol{y}
\\&=
\int_{(\R^{n})^{N}}
\big( \calF_{1,\dots,N} \sigma_{\boldsymbol{\nu}} \big) (x, y_1-x, \dots, y_N-x)
\prod_{j=1}^N \ichi_{B_{2R_{j}}}(x-y_{j}) \square_{\nu_{j}} f_{j} (y_{j}) 
\, d\boldsymbol{y},
\end{align*}
where, we wrote $d\boldsymbol{y}=dy_1 \cdots dy_N$. 
Then the Cauchy--Schwarz inequality
and the Plancherel theorem give
\begin{align*}
\left| T_{ \sigma_{\boldsymbol{\nu}} }( \square_{\nu_1}f_1, \dots, \square_{\nu_N}f_N )(x) \right|
\lesssim
\big\| \sigma_{\boldsymbol{\nu}} (x,\boldsymbol{\xi}) \big\|_{L^2_{\boldsymbol{\xi}}}
\prod_{j=1}^N
\left\{ \left( \mathbf{1}_{B_{2R_{j}}} \ast \big| \square_{\nu_{j}} f_{j} \big|^2 \right) (x) \right\}^{1/2}
\end{align*}
for any ${\boldsymbol{\nu}} = (\nu_1, \dots, \nu_N) \in \ZnN$ and $x \in \R^n$.
From this, the expression $I$ given in \eqref{decomposedFouriertransI} is estimated as
\begin{align}
\label{Ibeforediscretize}
\begin{split}
|I| \lesssim
\sum_{\boldsymbol{\nu}\in\ZnN} 
\int_{\R^n} & |\theta (x-\mu)| \,
\big\| \sigma_{\boldsymbol{\nu}} (x,\boldsymbol{\xi}) \big\|_{L^2_{\boldsymbol{\xi}}}
\prod_{j=1}^N
\left\{ \Big( \mathbf{1}_{B_{2R_{j}}} \ast \big|\square_{\nu_{j}} f_{j} \big|^2 \Big) (x) \right\}^{1/2}
\\
&\times 
\left| \varphi \left( \frac{D+ \nu_1 +\dots+ \nu_N }{R_0} \right) g (x) \right|
\, dx.
\end{split}
\end{align}

We next separate the integral by using \eqref{cubicdiscretization}.
Then the inequality \eqref{Ibeforediscretize} is identical with
\begin{align*}
|I| \lesssim
\sum_{\nu_0\in\Z^n} 
\sum_{\boldsymbol{\nu}\in\ZnN} 
\int_{Q} &|\theta (x+\nu_0-\mu)| \,
\big\| \sigma_{\boldsymbol{\nu}} (x+\nu_0,\boldsymbol{\xi}) \big\|_{L^2_{\boldsymbol{\xi}}}
\\
&\times
\prod_{j=1}^N
\left\{ \Big( \mathbf{1}_{B_{2R_{j}}} \ast \big|\square_{\nu_{j}} f_{j} \big|^2 \Big) (x+\nu_0) \right\}^{1/2}
\\
&\times 
\left| \varphi \left( \frac{D + \nu_1 +\dots+ \nu_N}
{R_0} \right) g (x+\nu_0) \right|
\, dx.
\end{align*}
Now, in this expression, 
we have $|\theta (x+\nu_0-\mu)| \lesssim \langle \nu_0-\mu \rangle^{-L}$ 
for any $x\in Q$ and $L> 0$ (we will choose a sufficiently large number 
$L$ in the forthcoming argument). 
Moreover, from Lemma \ref{propertiesofS} (5), (1), and (3), we have 
\begin{equation*}
\Big( \mathbf{1}_{B_{2R_{j}}} \ast \big|\square_{\nu_{j}} f_{j}\big|^2 \Big) (x+\nu_0)
\lesssim
S \Big( \mathbf{1}_{B_{2R_{j}}} \ast \big|\square_{\nu_{j}} f_{j} \big|^2 \Big) (\nu_0)
\end{equation*}
for $x\in Q$.
Hence,
\begin{align*}
|I| &\lesssim
\sum_{\nu_0\in\Z^n} 
\sum_{\boldsymbol{\nu}\in\ZnN} 
\langle \nu_0-\mu \rangle^{-L}
\prod_{j=1}^N
\left\{ S \Big( \mathbf{1}_{B_{2R_{j}}} \ast \big|\square_{\nu_{j}} f_{j} \big|^2 \Big) (\nu_0) \right\}^{1/2}
\\
&\quad \times 
\int_{Q} 
\big\| \sigma_{\boldsymbol{\nu}} (x+\nu_0,\boldsymbol{\xi}) \big\|_{L^2_{\boldsymbol{\xi}}}
\left| \varphi \left( \frac{D+\nu_1 +\dots+ \nu_N}{R_0} \right) g (x+\nu_0) \right|
\, dx.
\end{align*}
Using the Cauchy--Schwarz inequality to the integral on $x$, 
we obtain
\begin{align}
\label{Iafterdiscretize}
\begin{split}
|I|
&\lesssim
\sum_{\nu_0\in\Z^n} 
\sum_{\boldsymbol{\nu}\in\ZnN} 
\langle \nu_0-\mu \rangle^{-L}
\prod_{j=1}^N
\left\{ S \Big( \mathbf{1}_{B_{2R_{j}}} \ast \big|\square_{\nu_{j}} f_{j} \big|^2 \Big) (\nu_0) \right\}^{1/2}
\\
&\times
\big\| \sigma_{\boldsymbol{\nu}} (x+\nu_0,\boldsymbol{\xi}) \big\|_{L^2_{\boldsymbol{\xi}} L^2_x(Q)}
\left\| \varphi \left( \frac{D+\nu_1 +\dots+ \nu_N}{R_0} \right) g (x+\nu_0) \right\|_{L^2_x(Q)}.
\end{split}
\end{align}

Using the properties of the 
moderate function $W$, 
we shall prove 
\begin{equation}\label{vsigmaL2ul}
\sup_{\nu_0, \boldsymbol{\nu}}
\left\{
W (\boldsymbol{\nu})^{-1}
\big\| \sigma_{\boldsymbol{\nu}} 
(x+\nu_0,\boldsymbol{\xi}) 
\big\|_{L^2_{\boldsymbol{\xi}} L^2_x(Q)}
\right\}
\approx 
\|
W (\boldsymbol{\xi})^{-1} 
\sigma (x, \boldsymbol{\xi})
\|_{L^2_{ul} ((\R^n)^{N+1})} .
\end{equation}
Indeed, from the property 
\eqref{moderateboth} of moderate functions, 
we see that 
\begin{align*}
&
W (\boldsymbol{\nu})^{-1}
\big\| 
\sigma_{\boldsymbol{\nu}} (x+\nu_0,\boldsymbol{\xi}) 
\big\|_{L^2_{\boldsymbol{\xi}} L^2_x(Q)}
\\
&=
W (\boldsymbol{\nu})^{-1}
\left\| 
\sigma (x+\nu_0,\boldsymbol{\xi}) 
\ichi_{Q} (x) 
\prod_{j=1}^N \chi(\xi_j-\nu_j)
\right\|_{L^2_{ (x, \boldsymbol{\xi}) } (\R^{n(N+1)})}
\\
&
\lesssim
\left\| 
W (\boldsymbol{\xi})^{-1}
\sigma (x+\nu_0,\boldsymbol{\xi}) 
\ichi_{Q} (x) 
\langle (\boldsymbol{\xi}-\boldsymbol{\nu}) \rangle^{M}
\prod_{j=1}^N \chi(\xi_j-\nu_j)
\right\|_{L^2_{ (x, \boldsymbol{\xi}) } (\R^{n(N+1)})}
\\
&=
\left\|
W (\boldsymbol{\xi})^{-1}
\sigma (x,\boldsymbol{\xi})
\ichi _{Q}(x-\nu_0) 
\langle (\boldsymbol{\xi}-\boldsymbol{\nu}) \rangle^{M}
\prod_{j=1}^N \chi(\xi_j-\nu_j)
\right\|_{L^2_{ (x, \boldsymbol{\xi}) } (\R^{n(N+1)})}
=(\ast)
\end{align*}
for some $M>0$.
Here, 
from the fact $|\chi (y)|\ge c$ for $y\in Q$ and 
$\chi \in \calS$, we have 
\[
\ichi _{ [-1/2,1/2)^{(N+1)n} } (x,  \boldsymbol{\xi})
\lesssim 
\left|
\ichi _{Q}(x)
\langle \boldsymbol{\xi} \rangle^{M} 
\prod_{j=1}^N \chi(\xi_j) 
\right|
\lesssim 
\langle (x, \boldsymbol{\xi}) \rangle ^{-L}
\]
for any $L>0$.
Hence Lemma \ref{lemequivalentamalgam} implies that 
\[
\sup_{\boldsymbol{\nu}, \nu_0} 
(\ast) 
\approx 
\left\|
W (\boldsymbol{\xi})^{-1}
\sigma (x,\boldsymbol{\xi})
\right\|_{(L^2, \ell^{\infty} )(\R^{n(N+1)}) }
=
\left\|
W (\boldsymbol{\xi})^{-1}
\sigma (x,\boldsymbol{\xi})
\right\|_{L^2_{ul} (\R^{n(N+1)}) }. 
\]
Therefore we obtain the 
inequality $\lesssim $ of 
\eqref{vsigmaL2ul}.  
The opposite inequality $\gtrsim$ of \eqref{vsigmaL2ul} 
can be proved in a similar way.

Now, we get back to the estimate of \eqref{Iafterdiscretize}.
By virtue of \eqref{vsigmaL2ul}, 
we are able to derive from the inequality \eqref{Iafterdiscretize}
\begin{align}
\label{beforeAi}
\begin{split}
|I|
&\lesssim
\|
W (\boldsymbol{\xi})^{-1} 
\sigma (x, \boldsymbol{\xi})
\|_{L^2_{ul} ((\R^n)^{N+1})}
\\
&\quad\times
\sum_{\nu_0\in\Z^n} 
\langle \nu_0-\mu \rangle^{-L}
\sum_{\boldsymbol{\nu}\in\ZnN} 
W (\boldsymbol{\nu})
\prod_{j=1}^N
\left\{ S \Big( \mathbf{1}_{B_{2R_j}} \ast \big|\square_{\nu_j} f_j \big|^2 \Big) (\nu_0) \right\}^{1/2}
\\
&\quad\times
\left\| \varphi \left( \frac{D+ \nu_1 +\dots+ \nu_N}
{R_0} \right) g (x+\nu_0) \right\|_{L^2_x(Q)}.
\end{split}
\end{align}
In what follows, we simply write each summand by
\begin{align}\label{AjA0}
\begin{split}
A_j (\nu_j,\nu_0)&=\left\{ S \left( \mathbf{1}_{B_{2R_j}} \ast \big| \square_{\nu_j} f_j \big|^2 \right)(\nu_0) \right\}^{1/2},
\\
A_0(\mu,\nu_0)
&=
\left\|\varphi \left( \frac{D+\mu}{R_0} \right) g(x+\nu_0) \right\|_{L^2_x(Q)}
\end{split}
\end{align}
for $j=1,\dots,N$.
Then the inequality \eqref{beforeAi} is rewritten as
\begin{equation}
\label{IwithII}
|I| \lesssim
\|
W (\boldsymbol{\xi})^{-1} 
\sigma (x, \boldsymbol{\xi})
\|_{L^2_{ul} ((\R^n)^{N+1})}
\, \II
\end{equation}
with
\begin{equation}\label{II}
\II=
\sum_{\nu_0\in\Z^n} 
\langle \nu_0-\mu \rangle^{-L}
\sum_{\boldsymbol{\nu}\in\ZnN} 
W (\boldsymbol{\nu}) 
A_0(\nu_1 +\dots+ \nu_N, \nu_0)
\prod_{j=1}^N A_j (\nu_j,\nu_0).
\end{equation}

We shall estimate $\II$. 
Since the restriction of $W$ to $\ZnN$ belongs 
to the class $\calB_N (\ZnN)$,
we have
\begin{equation*}
\II
\lesssim 
\sum_{ \nu_0 \in \Z^{n} }
\langle \nu_0-\mu \rangle^{-L}
\| A_0 (\tau,\nu_0) \|_{\ell^2_{\tau}}
\prod_{j=1}^N 
\| A_j (\nu_j,\nu_0) \|_{\ell^2_{\nu_j}}. 
\end{equation*}
Using the H\"older inequality with 
the exponents $1=1/2 + \sum_{j=1}^N 1/(2N)$
to the above sum over $\nu_0$, we 
have 
\begin{equation}\label{IIHolder}
\II
\lesssim 
\| A_0 (\tau,\nu_0) \|_{ \ell^2_{\tau} \ell^{2}_{\nu_0} }
\prod_{j=1}^N 
\| \langle \nu_0-\mu \rangle^{-L/N} A_j (\nu_j,\nu_0) \|_{ \ell^2_{\nu_j} \ell^{2N}_{\nu_0}}. 
\end{equation}
Here, the norm of $A_0$ in \eqref{IIHolder}  
is estimated by the Plancherel theorem as follows:
\begin{align}\label{A0result}
\begin{split}
&\| A_0 (\tau,\nu_0) \|_{ \ell^2_{\tau} \ell^{2}_{\nu_0} }
=
\left\|\varphi \left( \frac{D+\tau}{R_0} \right) g(x+\nu_0) \right\|_{L^2_x(Q) \ell^2_{\tau} \ell^{2}_{\nu_0}}
\\
&=\left\|\varphi \left( \frac{D+\tau}{R_0} \right) g(x)
\right\|_{L^2_x(\R^n) \ell^2_{\tau}}
\approx
\left\|\varphi \left( \frac{\zeta +\tau}{R_0} \right) \widehat g(\zeta) \right\|_{L^2_\zeta(\R^n) \ell^2_{\tau}}
\approx 
R_{0}^{n/2} \| g \|_{L^2},
\end{split}
\end{align}
where we used that $\| \varphi ( \frac{\zeta +\tau}{R_0} ) \|_{\ell^2_\tau} \approx R_0^{n/2}$ 
for any $\zeta \in \R^n$ to have the last inequality.
Hence, by collecting \eqref{IwithII}, \eqref{IIHolder}, and \eqref{A0result},
we have
\begin{equation}
\label{resultofI}
|I|\lesssim
R_{0}^{n/2} \|
W (\boldsymbol{\xi})^{-1} 
\sigma (x, \boldsymbol{\xi})
\|_{L^2_{ul} ((\R^n)^{N+1})}
\, \| g \|_{L^2}
\, \prod_{j=1}^N 
\left\| \langle \nu_0-\mu \rangle^{-L/N} A_j (\nu_j,\nu_0) \right\|_{ \ell^2_{\nu_j} \ell^{2N}_{\nu_0}}.
\end{equation}

We substitute \eqref{resultofI} into \eqref{saishodual}, 
and then use the H\"older inequality for the 
quasi-norm $\ell^r_{\mu}$
with $1/r=\sum 1/q_j$
to obtain
\begin{align}\label{atodual}
\begin{split}
&\left\| T_{ \sigma }( f_1, \dots, f_N ) \right\|_{(L^2,\ell^r)}
\\
&\lesssim
R_{0}^{n/2} \|
W (\boldsymbol{\xi})^{-1} 
\sigma (x, \boldsymbol{\xi})
\|_{L^2_{ul}}
\, \prod_{j=1}^N 
\left\| \langle \nu_0-\mu \rangle^{-L/N} A_j (\nu_j,\nu_0) \right\|_{ \ell^2_{\nu_j} \ell^{2N}_{\nu_0} \ell^{q_j}_{\mu}}.
\end{split}
\end{align}

To achieve our goal, we shall estimate 
the norm of $A_j$ in \eqref{atodual}:
\begin{align*}
&
\left\| \langle \nu_0-\mu \rangle^{-L/N} A_j (\nu_j,\nu_0) \right\|_{ \ell^2_{\nu_j} \ell^{2N}_{\nu_0} \ell^{q_j}_{\mu}}
\\
&=
\bigg\| \langle \nu_0-\mu \rangle^{-2L/N} 
S \Big( \mathbf{1}_{B_{2R_j}} \ast 
\sum_{\nu_j} 
\big| \square_{\nu_{j}} f_{j} \big|^2 
\Big) (\nu_0) 
\bigg\|_{\ell^{N}_{\nu_0} \ell^{q_j/2}_{\mu} } 
^{1/2} 
=(\ast \ast).
\end{align*}
We use the inequality  
$\sum_{\nu_j} 
\big| \square_{\nu_{j}} f_{j} \big|^2 
\lesssim 
S (|f_j|^2)$ 
of Lemma \ref{estSRx} 
and also use Lemma \ref{propertiesofS} (1) and (2), 
change of variables, and the inequality 
\eqref{mixedtriangle} to obtain 
\begin{align*}
(\ast \ast )
&\lesssim
\bigg \| \langle \nu_0 \rangle^{-2L/N}
S \Big( \mathbf{1}_{B_{2R_j}} \ast | f_j |^2 
\Big)(\nu_0+\mu) \bigg\|_{ \ell_{\nu_0}^{N} \ell^{q_j/2}_{\mu} }^{1/2}
\\
&
\lesssim
\bigg \| 
S \Big( \mathbf{1}_{B_{2R_j}} \ast | f_j |^2 
\Big)(\mu) \bigg\|_{ \ell^{q_j/2}_{\mu} }^{1/2}, 
\end{align*}
where the 
latter $\lesssim$ holds if $L$ 
is suitably large. 
Thus 
\begin{equation}\label{tochuukeisanAj}
\left\| \langle \nu_0-\mu \rangle^{-L/N} 
A_j (\nu_j,\nu_0) \right\|_{ \ell^2_{\nu_j} \ell^{2N}_{\nu_0} \ell^{q_j}_{\mu}}
\lesssim
\left\| 
S \left( \mathbf{1}_{B_{2R_j}} \ast | f_j |^2 \right)
(\mu) \right\|_{ \ell^{q_j/2}_{\mu} }^{1/2}.
\end{equation}
Here, since $B_R$, $R\geq1$, is covered by a disjoint union of the cubes 
$\tau + Q$, $\tau \in \Z^n \cap [-R-1,R+1]^n$, 
the characteristic function $\ichi_{B_R}$ 
can be bounded by
\begin{equation*}
\ichi_{B_R} \leq
\sum_{\tau \in \Z^n \cap [-R-1,R+1]^n} \ichi_{\tau+Q}
=
\sum_{\tau \in \Z^n \cap [-R-1,R+1]^n} \ichi_{Q}(\cdot-\tau).
\end{equation*}
This yields 
\begin{align*}
S \left( \mathbf{1}_{B_{2R_j}} \ast | f_j |^2 \right)(\mu)
\lesssim
\sum_{|\tau|\lesssim R_j}
S \left( \mathbf{1}_{Q} \ast | f_j |^2 \right)(\mu - \tau)
\approx
\left\| S \left( | f_j |^2 \right)(\mu - \tau) \right\|_{\ell^1(\{|\tau|\lesssim R_j\})} ,
\end{align*}
where the last $\approx$ follows from 
Lemma \ref{propertiesofS} (4). 
Thus, by \eqref{mixedtriangle} and Lemma \ref{SellequivL} (1) and (2), 
\begin{align*}
&\left\| S \left( \mathbf{1}_{2B_{R_j}} 
\ast | f_j |^2 \right)(\mu) \right\|_{ \ell^{q_j/2}_{\mu} }
\lesssim
\left\|
S \left( | f_j |^2 \right)(\mu - \tau) 
\right\|_{ \ell^{1}(\{|\tau| \lesssim R_j \}) \, \ell^{q_j/2}_{\mu} }
\\
&
\le 
\left\|
S \left( | f_j |^2 \right)(\mu - \tau) 
\right\|_{ \ell^{q_j/2}_{\mu} 
\ell^{ \min (1, q_j/2 )}(\{|\tau| \lesssim R_j\}) }
\approx 
R_j^{n/\min(1,q_j/2)}\left\| 
S \left( | f_j |^2 \right)(\mu) 
\right\|_{ \ell^{q_j/2}_{\mu} }
\\
&\approx 
R_j^{n/\min(1,q_j/2)} \| f_j \|_{(L^2,\ell^{q_j})}^2, 
\end{align*}
which combined with 
\eqref{tochuukeisanAj} gives 
\begin{equation}\label{estofAj}
\left\| \langle \nu_0-\mu \rangle^{-L/N} 
A_j (\nu_j,\nu_0) \right\|_{ \ell^2_{\nu_j} \ell^{2N}_{\nu_0} \ell^{q_j}_{\mu}}
\lesssim
R_j^{n/\min(2,q_j)} \| f_j \|_{(L^2,\ell^{q_j})}.
\end{equation}

Substituting \eqref{estofAj} into \eqref{atodual}, we obtain
\begin{align*}
&\left\| T_{ \sigma }( f_1, \dots, f_N ) \right\|_{(L^2,\ell^r)}
\\
&\lesssim
R_{0}^{n/2} \left( \prod_{j=1}^N 
R_j^{n/\min(2,q_j)} \right) 
\|
W (\boldsymbol{\xi})^{-1} 
\sigma (x, \boldsymbol{\xi})
\|_{L^2_{ul}}
\prod_{j=1}^N  \| f_j \|_{(L^2,\ell^{q_j})},
\end{align*}
which completes the proof.
\end{proof}

\subsection{A theorem for symbols with 
limited smoothness}
\label{subsectionMainTh}

From Proposition \ref{main-prop}, 
we shall deduce a theorem concerning 
the multilinear pseudo-differential 
operators $T_{\sigma}$ with 
symbols of limited smoothness. 
To measure the smoothness of such symbols, 
we shall use Besov type norms.  
To define the Besov type norms, 
we use the partition of unity given as follows.
Take a $\phi \in \calS (\R^n)$ 
such that $\phi (y)= 1$ for 
$ | y | \leq 1 $  
and 
$\supp \phi \subset 
\{ y \in \R^n : | y | \leq 2 \}$.  
We put $\psi (y)= \phi (y) - \phi (2y)$. 
Then 
$\supp \psi \subset 
\{ y \in \R^n : 1/2 \leq | y | \leq 2 \}$.
We set 
$\psi_0 = \phi $ and 
$\psi_k = \psi (\cdot / 2^k )$ for $k \in \N$.
Then $\sum_{k=0}^{\infty} \psi_{k} (y) = 1$ 
for all $y \in \R^n$. 
We shall call $\{\psi_k\}_{k\in \N_0}$ a  
Littlewood--Paley partition of unity on $\R^n$.  
It is easy to see that 
the Besov type norms given in the following definition 
do not depend, up to the equivalence of norms, 
on the choice of Littlewood--Paley partition of unity.

\begin{defn}\label{defBSW2} 
Let $W\in \calM (\R^{Nn})$. 
Let $\{\psi_{k}\}_{ k\in \N_0 }$ 
be a Littlewood--Paley 
partition of unity on $\R^n$.
For 
\begin{align*}
\boldsymbol{k} 
&=
(k_0, k_1, \dots, k_N) \in (\N_0)^{N+1}, 
\\
\boldsymbol{s}
&=
(s_0, s_1, \dots, s_N) \in [0, \infty)^{N+1}, 
\end{align*}
and
$\sigma = \sigma (x, \xi_1, \dots, \xi_N) 
\in L^{\infty} 
(\R^n_{x}\times 
\R^n_{\xi_1}\times \dots \times
\R^n_{\xi_N})$, 
we write 
$\boldsymbol{s}\cdot 
\boldsymbol{k}
=
\sum_{j=0}^{N}
s_{j} k_{j}$ and
\begin{equation*}
\Delta_{\boldsymbol{k}} 
\sigma (x, \xi_1, \dots, \xi_N)
=
\psi_{k_0} ( D_x )
\psi_{k_1} ( D_{\xi_1} )
\cdots
\psi_{k_N} ( D_{\xi_N} )
\sigma (x, \xi_1, \dots, \xi_N).
\end{equation*}
We denote by $S^{W}_{0,0} (\boldsymbol{s}, t ; \R^n, N)$ 
for $t \in (0, \infty]$
the set of all $\sigma \in L^{\infty} ((\R^n)^{N+1})$ for which 
the following quasi-norm is finite: 
\begin{align*}
&
\|\sigma\|_{ S^{W}_{0,0} (\boldsymbol{s},t; \R^n, N) } 
\\
&=
\left\{
 \sum_{ \boldsymbol{k} \in (\N_0)^{N+1}}\, 
 \left( 2^{ \boldsymbol{s}\cdot \boldsymbol{k} } \right)^{t}
\big\|
W (\xi_1, \dots, \xi_N)^{-1}
\Delta_{\boldsymbol{k}} 
\sigma (x, \xi_1, \dots, \xi_N)
\big\|_{ L^2_{ul} ((\R^{n})^{N+1}) }^t \right\}^{1/t}
\end{align*}
with usual modification when $t=\infty$.
\end{defn}

In terms of these notations, 
the theorem reads as follows.  

\begin{thm}\label{main-thm-2} 
Let 
$W \in \calM (\R^{Nn})$ 
and suppose 
the restriction of $W$ to $(\Z^n)^N$ belongs 
to the class $\calB_N ((\Z^n)^N)$. 
Let  $q_j, r \in (0,\infty]$ and $s_0, s_{j} \in [0,\infty)$,
$j=1,\dots,N$,
satisfy 
$\sum_{j=1}^N 1/q_j = 1/ r$, 
$s_0 = n/2$, and
$s_{j} = n/\min(2, q_{j})$.
Then 
the multilinear  
pseudo-differential operator 
$T_{\sigma}$ for 
$\sigma \in 
S^{W}_{0,0} (\boldsymbol{s},\min(1,r); \R^n, N)$ 
is bounded from 
$(L^2,\ell^{q_1}) (\R^n) \times \cdots \times (L^2,\ell^{q_N}) (\R^n)$ 
to 
$(L^2, \ell^r)(\R^n)$. 
\end{thm}

\begin{proof}
Theorem \ref{main-thm-2} can be reduced to
Proposition \ref{main-prop} as follows. 
We decompose the symbol $\sigma$ 
by using the Littlewood--Paley partition: 
\begin{equation*}
\sigma (x, \xi_1, \dots, \xi_N)
=
\sum_{ \boldsymbol{k} \in (\N_0)^{N+1}} 
\Delta_{\boldsymbol{k}} 
\sigma (x, \xi_1, \dots, \xi_N). 
\end{equation*}
Then the 
support of 
$\calF (\Delta_{\boldsymbol{k}} \sigma )$ 
is included in 
$B_{R_0} \times B_{R_1} \times \cdots \times B_{R_N}$
with 
$R_{j}= 2^{ k_{j} +1}$,
$j=0,1,\dots,N$,
so that, Proposition \ref{main-prop} 
yields
\begin{align*}
&
\|T_{ \Delta_{\boldsymbol{k}} \sigma }
\|_{ (L^2,\ell^{q_1}) \times \cdots \times (L^2,\ell^{q_N}) \to (L^2, \ell^r)}
\\
&\lesssim 
\bigg(
\left(2^{ k_{0} } \right)^{n/2}
\prod_{j=1}^{N}
\left(2^{ k_{j} } \right)^{ n/\min(2,q_j) }
\bigg) 
\big\| W (\xi_1, \dots, \xi_N)^{-1}
 \Delta_{\boldsymbol{k}} \sigma 
(x, \xi_1, \dots, \xi_N) 
\big\|_{ L^2_{ul} }. 
\end{align*} 
Taking sum over ${ \boldsymbol{k} \in (\N_0)^{N+1}} $, 
we obtain  
\begin{align*}
\|T_{ \sigma }
\|_{ (L^2,\ell^{p_1}) \times \cdots \times (L^2,\ell^{p_N}) \to (L^2, \ell^r)}^{\min(1,r)}
&\le 
\sum_{ \boldsymbol{k} \in (\N_0)^{N+1}} 
\|T_{ \Delta_{\boldsymbol{k}} \sigma }
\|_{ (L^2,\ell^{p_1}) \times \cdots \times (L^2,\ell^{p_N}) \to (L^2, \ell^r)}^{\min(1,r)}
\\
&
\lesssim 
\sum_{ \boldsymbol{k} \in (\N_0)^{N+1}} 
\bigg(
(2^{ k_{0} })^{n/2}
\prod_{j=1}^{N}
(2^{ k_{j} })^{ n/\min(2,q_j) }
\bigg) ^{\min(1,r)}
\\
&\quad\times
\big\| W (\xi_1, \dots, \xi_N)^{-1}
 \Delta_{\boldsymbol{k}} \sigma 
(x, \xi_1, \dots, \xi_N) 
\big\|_{ L^2_{ul} }^{\min(1,r)}
\\
&=
\|\sigma\|_{ S^{W}_{0,0} 
( \boldsymbol{s},\min(1,r); \R^n, N) } ^{\min(1,r)}
\end{align*}
with $s_0 = n/2$ and
$s_{j} = n/\min(2, q_{j})$, 
$j=1,\dots,N$,
which is the desired result. 
\end{proof}

\subsection{Symbols with classical derivatives} 
\label{subsectionClassical}

The following proposition shows  
that symbols that have classical derivatives 
up to certain order satisfy the conditions of 
Theorem \ref{main-thm-2}.

\begin{prop}
\label{classicalderivative} 
Let $\boldsymbol{s}=(s_0, s_1, \dots, s_N) \in [0, \infty)^{N+1}$. 
Let $\sigma = \sigma (x, \xi_1, \dots, \xi_N)$ 
be a bounded measurable 
function on $(\R^n)^{N+1}$ 
and $W \in \calM ( \R^{Nn} )$. 
Suppose 
\[
 |
\partial_{x}^{ \alpha_{0} }
\partial_{ \xi_{1} }^{ \alpha_{1} } \cdots
\partial_{ \xi_{N} }^{ \alpha_{N} }
\sigma (x, \xi_1, \dots, \xi_N)
|
\le 
W (\xi_1, \dots, \xi_N) 
\]
for 
$\alpha_{j}\in (\N_0)^n$ 
with 
$|\alpha_{j}| \le [s_{j}]+1$. 
Then 
$\sigma \in 
S^{W}_{0,0} ( \boldsymbol{s}, t ; \R^n, N)$ 
for every $t \in (0, \infty]$. 
To be precise, 
the above assumptions should be understood that 
the derivatives of $\sigma$ taken 
in the sense of distribution 
are locally integrable functions on $(\R^{n})^{N+1}$ 
and they are bounded by $W(\xi_1, \dots, \xi_N)$ 
almost everywhere.  
\end{prop}

This proposition is proved in 
\cite[Proposition 4.7]{KMT-arXiv} 
for the case $N=2$ and $t=1$. 
The proof can be applied to the general case $N\ge 1$ 
and $t \in (0, \infty]$ with obvious modifications. 
Thus we omit the proof here.

\subsection{Proof of Theorem \ref{main-thm-1}}
\label{subsection-Proof-Main-1}

Here we give a proof of
Theorem \ref{main-thm-1}.

\begin{proof} 
We prove the assertion (2) first. 
Suppose $V \in \calB_N (\ZnN)$ 
and $\sigma \in 
S^{ \widetilde{V} }_{0,0} (\R^n, N)$. 
We take a function $V^{\ast}$ as mentioned in 
Proposition \ref{Vast}. 
Since $V^{\ast}$ is a moderate function, 
it follows that $\widetilde{V} \lesssim V^{\ast}$ 
and hence $\sigma \in S^{ V^{\ast} }_{0,0} (\R^n, N)$. 
Proposition \ref{classicalderivative} 
implies that 
$\sigma$ also satisfies  
the assumptions of Theorem \ref{main-thm-2} 
with $W = V^{\ast}$. 
Hence 
we obtain the boundedness of $T_{\sigma}$ for 
the case $\sum_{j}1/q_j =1/r$. 
Since the embedding 
$(L^2, \ell^r) \hookrightarrow (L^2, \ell^s) $ 
holds for $r \le s$, 
the boundedness also holds for 
the case $\sum_{j}1/q_j \ge 1/r$.

Next, we shall prove the assertion (1). 
The basic idea of this part
goes back to \cite[Proof of Lemma 6.3]{MT-2013}.

Let $V$ be a nonnegative bounded function on 
$\ZnN$ and $0<q_j, r \leq \infty$, $j=1,\dots,N$. 
We assume 
$\mathrm{Op} (S^{\widetilde{V}}_{0,0}(\R^n, N) ) 
\subset 
B \big( (L^2,\ell^{q_1} ) \times \dots \times (L^2,\ell^{q_N}) \to (L^2,\ell^{r} ) \big)$ 
with 
$\widetilde{V}$ defined as in 
Theorem \ref{main-thm-1}. 
By the closed graph theorem, 
it follows that 
there exist a positive integer $M$
and a positive constant $C$ such that
\begin{align}\label{inequality001}
\begin{split}
&
\|T_{\sigma}\|_{ (L^2,\ell^{q_1} ) \times \dots \times (L^2,\ell^{q_N}) \to (L^2,\ell^{r} ) } 
\\
&\le 
C\max_{|\alpha|, |\beta_1|,\dots |\beta_N| \le M}
\left\| 
\widetilde{V}(\xi_1,\dots,\xi_N)^{-1} 
\partial^{\alpha}_x 
\partial^{\beta_1}_{\xi_1} \cdots \partial^{\beta_N}_{\xi_N} 
\sigma(x,\xi_1,\dots,\xi_N)
\right\|_{L^{\infty}}
\end{split}
\end{align}
for all bounded smooth functions $\sigma$ on $(\R^n)^{N+1}$ 
(as for the argument using closed graph theorem, 
see \cite[Lemma 2.6]{BBMNT}). 
Our purpose is to prove the inequality 
\eqref{BL222}. 
To this end,
by a limiting argument,
it is sufficient to consider 
$A_{j} \in \ell^2 (\Z^n)$ such that 
$A_{j}(\mu)=0$ except for 
a finite number of $\mu \in \Z^n$,
$j=0,1,\dots,N$.

Take $\varphi,\widetilde{\varphi} \in \calS(\R^n)$ 
such that 
\begin{align}
&\mathrm{supp}\, \widetilde{\varphi} 
\subset [-1/2,1/2]^n, 
\quad
\widetilde{\varphi}=1 \ \text{on $[-1/4,1/4]^n$},  
\quad 
\mathrm{supp}\, \varphi \subset [-1/4,1/4]^n, 
\nonumber
\\
&
|\calF ^{-1} \varphi | \ge 1 \;\; 
\text{on $[-\pi,\pi ]^n$}. 
\label{varphi-live}
\end{align}
Set
\begin{equation*}
\sigma (x,\xi_1,\dots,\xi_N)
=\sigma (\xi_1,\dots,\xi_N)
=\sum_{ k_1,\dots,k_N \in \Z^n }
V(k_1, \dots, k_N) 
\prod_{j=1}^N \widetilde{\varphi}(\xi_{j}-k_{j}),
\end{equation*}
and define $f_{j} \in \calS (\R^n)$, $j=1,\dots,N$, by 
\begin{equation*}
f_{j} (x)
=
\sum_{\nu_{j} \in \Z^n}
A_{j}(\nu_{j})
e^{i \nu_{j} \cdot x} 
\calF^{-1}\varphi (x).
\end{equation*}

Then 
\begin{equation}\label{estimate-sigma}
| \partial^{\beta_1}_{\xi_1} \cdots \partial^{\beta_N}_{\xi_N} 
\sigma(\xi_1,\dots,\xi_N) |
\le C_{\beta_1,\dots, \beta_N}
\widetilde{V} (\xi_1, \dots, \xi_N).
\end{equation}
For $f_j$, we have 
\begin{equation}\label{fiamalgam}
\| f_{j} \|_{(L^2,\ell^{q_j})} \approx \| A_{j} \|_{\ell^2}
\end{equation}
for each $0< q_j \leq \infty$. 
In fact, 
notice that  
$\sum_{\nu_j } A_j (\nu_j) e^{i \nu_j \cdot x}$ 
is a $(2\pi \Z)^n$ periodic function and 
Parseval's identity gives 
\[
\left\|
\sum_{\nu_j } A_j (\nu_j) 
e^{i \nu_j \cdot (x+2 \pi \nu ) }
\right\|_{L^2 _{x} ([-\pi, \pi]^n)} 
\approx 
\|A_j\|_{\ell^2}
\]
for all $\nu \in \Z^n$.  
Thus using the property 
\eqref{varphi-live} and the fact that 
$\calF^{-1} \varphi $ is rapidly decreasing, 
we have 
\begin{align*}
\| f_{j} \|_{(L^2,\ell^q)} 
\approx 
\| f_{j} (x + 2 \pi \nu )
\|_{L^2_{x} ([-\pi, \pi]^n)\ell^q_{\nu} (\Z^n)} 
\approx 
\|A_j\|_{\ell^2}
\end{align*}
as desired (the proof of the former 
$\approx$ in the above inequalities is easy and is 
left to the reader).

Since
$\widehat{f_{j}}(\xi_{j} )
=\sum_{\nu_{j} \in \Z^n}
A_{j}(\nu_{j})
\varphi(\xi_{j} - \nu_{j})$,
by the conditions on 
$\varphi$ and 
$\widetilde{\varphi}$, 
we have 
\begin{align*}
T_{\sigma}(f_1,\dots,f_N)(x)
&=
\sum_{ \nu_1, \dots, \nu_N \in \Z^n }
V(\nu_1, \dots, \nu_N)
\left(\prod_{j=1}^N A_{j}(\nu_{j})
\right) 
e^{i (\nu_1 +\dots+ \nu_N)\cdot x} 
\calF^{-1} \varphi (x)^{N}
\\
&
=\sum_{k}
d_k e^{i k \cdot x} 
\calF^{-1} \varphi (x)^{N},
\end{align*}
where
\begin{equation}\label{def-dk}
d_k = 
\sum_{\nu_1 +\dots+ \nu_N =k} 
V(\nu_1, \dots, \nu_N)
\prod_{j=1}^N A_{j}(\nu_{j}). 
\end{equation}
Notice that 
$d_k\neq 0$ only for a finite number of $k$'s  
by virtue of our assumptions on $A_{j}$, $j=1,\dots,N$.  
Then, by the same reason as in \eqref{fiamalgam}, 
we have 
\begin{equation}
\label{Tsigmaamalgam}
\| T_{\sigma}(f_1,\dots,f_N) \|_{(L^2,\ell^r)}
\approx
\| d_k \|_{\ell^2_k}
\end{equation}
for each $0< r \leq \infty$.

Combining \eqref{inequality001}, 
\eqref{estimate-sigma},
\eqref{fiamalgam}, \eqref{def-dk}, 
and \eqref{Tsigmaamalgam}, 
we obtain 
\[
\bigg\|
\sum_{\nu_1 +\dots+ \nu_N =k} 
V(\nu_1, \dots, \nu_N)
\prod_{j=1}^N A_{j}(\nu_{j})
\bigg\|_{\ell^2_k}
\lesssim 
\prod_{j=1}^N 
\| A_{j}\|_{\ell^2},
\]
which is equivalent to \eqref{BL222}. 
This completes the proof of 
Theorem \ref{main-thm-1}. 
\end{proof}


\section{Boundedness in the $L^{p}$ framework}\label{sectionBddinLp}

In this section, we shall consider boundedness 
of multilinear pseudo-differential opeartors 
under sharp regularity conditions.

The boundedness 
of linear and multilinear pseudo-differential operators
for symbols with limited  smoothness
have been studied by several researchers.
For instance, 
results for linear pseudo-differential operators
were obtained by
Cordes \cite{Cordes}, 
Coifman--Meyer \cite{CM}, 
Muramatu \cite{Muramatu}, 
Miyachi \cite{Miyachi}, 
Sugimoto \cite{Sugimoto}, 
and Boulkhemair \cite{boulkhemair 1995}, 
and results for bilinear operators
were obtained by 
Herbert--Naibo \cite{herbert naibo 2014, herbert naibo 2016}.
For more than 3-fold 
multilinear pseudo-differential operators,
the present authors cannot find related results.
For the linear case,
the above mentioned authors proved that,
roughly speaking, 
smoothness of symbols up to 
$n/2$ for each variable $x$ and $\xi$ 
assures the boundedness in $L^2$. 
For the bilinear case, 
in \cite[Theorem 1.1]{herbert naibo 2016},
the authors proved that 
the bilinear pseudo-differential operators
with $x$-independent symbols 
of class $S^{\langle m \rangle }_{0,0}(\R^n,2)$ 
with $m < -n/2$  
are bounded from $L^2 \times L^2$ to $L^1$ 
if the smoothness 
up to $n$ for the $\xi_1$ and $\xi_2$ variables 
are assumed.
Moreover, in \cite[Theorem 2]{herbert naibo 2014},
the authors proved that
the bilinear operators of class 
$S^{\langle m \rangle }_{0,0}(\R^n,2)$ 
with $m < -n/2$
are bounded from $L^2 \times L^\infty$ 
(and $L^\infty \times L^2$) to $L^2$ 
if the smoothness up to $n/2$ for the 
$x$, $\xi_1$, and $\xi_2$ variables are assumed.
Notice that these results of 
\cite{herbert naibo 2014} 
and \cite{herbert naibo 2016} 
locate at the endpoints of 
the range $1 \leq r \leq 2 \leq q_1,q_2 \leq \infty $ with $1/q_1+1/q_2=1/r$
noted in \eqref{bilinearLpbdd}.

The purpose of this section is the following. 
Firstly, for linear and bilinear cases,
we generalize the $L^p$-boundedness 
to the $(L^2,\ell^p)$-boundedness
as in Theorem \ref{main-thm-1}.
Secondly, we relax 
the assumptions on $m$ and 
on the regularity of the symbols 
given in 
\cite[Theorem 2]{herbert naibo 2014} and 
\cite[Theorem 1.1]{herbert naibo 2016}. 
Thirdly, 
we generalize the results  
to more than 3-fold multilinear case.

The main theorem of this section is the following. 
This theorem is a generalization of 
\cite[Theorem 4.5]{KMT-arXiv}, 
where the case 
$N=2$, $q_j=2$, and $r \in [1,2]$
is given.

\begin{thm}\label{main-thm-3} 
Let 
$W \in \calM (\R^{Nn})$ 
and suppose 
the restriction of 
$W$ to $(\Z^n)^N$ belongs 
to the class $\calB_N ((\Z^n)^N)$. 
Let $q_j \in [2,\infty]$, $r \in [2/N,\infty]$, 
$s_0, s_j \in [0,\infty)$, $j=1, \dots, N$, 
and suppose 
$\sum_{j=1}^{N}1/q_j \ge 1/r$ and 
\begin{equation*}
s_0 = \frac{n}{2}, 
\quad
\frac{n}{2} - \frac{n}{q_j} \leq s_j \leq \frac{n}{2},
\quad 
\sum_{j=1}^N s_j
=
\sum_{j=1}^N \left( \frac{n}{2} - \frac{n}{q_j} \right) 
+
\frac{n}{r} .
\end{equation*}
Then 
the multilinear 
pseudo-differential operator 
$T_{\sigma}$ for 
$\sigma \in 
S^{W}_{0,0} (\boldsymbol{s},\min(1,r); \R^n, N)$ 
is bounded from 
$(L^2, \ell^{q_1} )\times \cdots \times (L^2, \ell^{q_N}) $ 
to 
$(L^2, \ell^r)$. 
If in addition $2/N\le r \leq 2$, 
then 
$T_\sigma$ is bounded from
$L^{q_1} \times \cdots \times L^{q_N} $ 
to 
$h^{r} $,  
where $L^{q_j}$ can be replaced by $bmo$ 
when $q_j=\infty$.
\end{thm}

We omit the detailed comparison 
of Theorem \ref{main-thm-3}
and the results 
mentioned at the beginning of this section. 
We only note that 
in the case $N=2$, 
$W (\nu_1, \nu_2) = (1+ |\nu_1|+|\nu_2|)^{-n/2}$, 
$s_0=s_1=s_2=n/2$, 
and 
$q_1=q_2=2$, $r=1$  
(resp.\ $q_1 =r=2$, $q_2 = \infty$), 
Theorem \ref{main-thm-3} implies
$\op(S_{0,0}^{\langle -n/2 \rangle }
(\boldsymbol{s},1; \R^n, 2))
\subset B(L^2 \times L^2 \to h^1)$ 
(resp.\ $\op(S_{0,0}^{\langle -n/2 \rangle}
(\boldsymbol{s},1; \R^n, 2))
\subset B(L^2 \times bmo \to L^2)$), 
which is an improvement 
of 
\cite[Theorem 1.1]{herbert naibo 2016} 
(resp.\ \cite[Theorem 2]{herbert naibo 2014}).

Observe that in the case 
$\sum_{j=1}^{N}1/q_j = 1/r$ of Theorem 
\ref{main-thm-3} the only admitted choice of 
$s_j$ is 
$s_j = n/2$ and this is the same 
as $s_j$ given in Theorem \ref{main-thm-2}. 
In the case $\sum_{j=1}^{N}1/q_j > 1/r$, however, 
Theorem \ref{main-thm-3} admits some smaller $s_j$. 
In the next section, we shall prove that the conditions 
on $s_0$ and $s_j$ given in Theorem \ref{main-thm-3} 
are sharp.

Now, we shall prove Theorem \ref{main-thm-3}. 
By the same reasoning as in Subsection \ref{subsectionMainTh}, 
Theorem \ref{main-thm-3} will be derived from the proposition given below. 

\begin{prop}
\label{main-prop2}
Let $W \in \calM (\R^{Nn})$ and suppose 
the restriction of $W$ to $\ZnN$ belongs 
to the class $\calB_N (\ZnN)$. 
Let 
$R_{0}, R_{j} \in [1, \infty)$,
$p_{j}, q_{j} \in [2, \infty]$, $j=1,\dots, N$, 
and $r \in (0,\infty]$ satisfy 
\begin{equation}\label{assumptionpqr}
\frac{1}{p_{j}}+\frac{1}{q_{j}}-\frac{1}{2} \geq 0 
\quad and \quad
\sum_{j=1}^N \left(\frac{1}{p_{j}}+\frac{1}{q_{j}}-\frac{1}{2}\right) = \frac{1}{r}.
\end{equation} 
Suppose $\sigma$ is 
a bounded continuous function on 
$(\R^n)^{N+1}$ such that  
$\supp \calF \sigma \subset 
B_{R_0} \times B_{R_1} \times \cdots \times B_{R_N}$. 
Then 
\begin{equation}\label{conclusionProp2}
\begin{split}
&
\left\| T_{ \sigma } \right\|_{(L^2,\ell^{q_1}) \times \dots \times (L^2,\ell^{q_N}) \to (L^2, \ell^{r})}
\\
&\lesssim
R_{0}^{n/2}
\bigg(\prod_{j=1}^N R_{j}^{ n/p_{j} }\bigg) 
\big\|
W(\xi_1, \dots, \xi_N)^{-1} 
\sigma (x,\xi_1,\dots,\xi_N) \big\|_{ L^{2}_{ul} ((\R^{n})^{N+1})}. 
\end{split}
\end{equation}
\end{prop}

This proposition is a generalization of 
\cite[Proposition 4.1]{KMT-arXiv}, 
where the case $N=2$, $q_j=2$, and $r \in [1,2]$ 
is given. 
Notice that if we choose $p_j=2$, $j=1, \dots, N$, 
then the claim of the proposition coincides with 
that of Proposition \ref{main-prop}.

Here is the proof of the proposition. 

\begin{proof}[Proof of Proposition \ref{main-prop2}]
We repeat the same argument as in the proof of Proposition \ref{main-prop}. 
If $r_j \in (0, \infty]$ satisfy $1/r = \sum_{j=1}^{N} 1/ r_j$, 
then 
in the same way as we obtained \eqref{atodual} 
and \eqref{tochuukeisanAj}, we obtain 
\begin{align*}
\begin{split}
&\left\| T_{ \sigma }( f_1, \dots, f_N ) \right\|_{(L^2,\ell^r)}
\\
&\lesssim
R_{0}^{n/2} \|
W (\xi_1,\dots,\xi_N)^{-1} 
\sigma (x, \xi_1,\dots,\xi_N)
\|_{L^2_{ul}}
\, \prod_{j=1}^N 
\left\|\mathbf{1}_{B_{2R_j}} \ast S \left( | f_j |^2 \right) \right\|_{ L^{r_j/2} }^{1/2}.
\end{split}
\end{align*}
We take $1/r_j = 1/p_j +1/q_j -1/2$ with $p_j$ and $q_j$ given in 
the proposition. 
Then the assumptions of the proposition imply 
that $p_j/2, q_j/2 \in [1, \infty]$ and 
\[
0\le \frac{2}{r_{j}} = \frac{2}{p_{j}} + \frac{2}{q_{j}} - 1
\le 1.  
\]
Hence applying the Young inequality 
and using Lemma \ref{SellequivL} (2), 
we have 
\begin{equation*}
\left\|\mathbf{1}_{B_{2R_j}} \ast S \left( | f_j |^2 \right) \right\|_{ L^{r_j/2} }^{1/2}
\le
\left(\|\mathbf{1}_{B_{2R_j}}\|_{L^{p_j/2}}
\|S \left( | f_j |^2 \right)\|_{L^{q_j/2}}\right)^{1/2}
\lesssim
R_j^{n/p_j} \| f_j \|_{(L^2,\ell^{q_j})}. 
\end{equation*}
Combining the above inequalities, we obtain the desired result.
\end{proof}

\begin{proof}[Proof of Theorem \ref{main-thm-3}]
By the same argument as in Proof of 
Theorem \ref{main-thm-2}, 
the statement concerning the boundedness of 
$T_{\sigma}$ in $L^2$-based amalgam spaces follows 
from Proposition \ref{main-prop2} with $s_j = n/p_j$, $j=1,\dots,N$. 
The claim concerning the boundedness of 
$T_{\sigma}$ from 
$L^{q_1} \times \cdots \times L^{q_N}$ to 
$h^r$ follows from the embeddings of the spaces 
given in Subsection \ref{secamalgam}. 
\end{proof}


\section{Sharpness of Theorem \ref{main-thm-3}}\label{sectionSharpness}

In this section,
we shall discuss sharpness of some 
conditions for the boundedness of multilinear pseudo-differential 
operators in Lebesgue spaces.  
To do this, we consider 
the following special weight function 
\begin{equation*}
W_m(\xi_1,\dots,\xi_N)
=(1+|\xi_1|+\dots+|\xi_N|)^{m},
\quad m \in (-\infty,0],
\end{equation*}
and denote the class 
$S^{W_m}_{0,0} (\boldsymbol{s},t; \R^n, N)$ 
by 
$S^{\langle m \rangle }_{0,0}(\boldsymbol{s},t; \R^n, N)$.  
Recall that, as we have already defined in Introduction,  
$S^{W_m}_{0,0}(\R^n, N)
=S^{\langle m \rangle }_{0,0}(\R^n, N)$. 
Also recall that the restriction to $\ZnN$ of 
$W_m$ with $m= -(N-1)n/2$ is an example of a weight in 
$\calB_{N} (\ZnN)$ (see Example \ref{example-of-V}).  
We shall prove that 
the number  
$m = -(N-1)n/2$ is a critical one and 
also prove that the conditions set on 
$q_j$, $r$, and $s_j$ in Theorem \ref{main-thm-3} 
are sharp.

We use the following fact due to Wainger \cite[Theorem 10]{Wainger}.
Let $0<a<1$ and $0<b<n$.
For $t>0$, define
\begin{equation*}
\widetilde{f}_{a,b,t}(x)
=\sum_{k \in \Z^n \setminus \{0\}}
e^{-t|k|}|k|^{-b}e^{i|k|^a}e^{i k\cdot x}.
\end{equation*}
Then the limit
$\widetilde{f}_{a,b}(x)=\lim_{t \to +0}
\widetilde{f}_{a,b,t}(x)$
exsits for all $x \in \R^n \setminus (2\pi \Z)^n$.
If in addition $1 \le q \le \infty$ and 
$b>n-an/2-n/q+an/q$, then $\widetilde{f}_{a,b}$
is a function in $L^q(\T^n)$ whose Fourier coefficients 
are given by 
\begin{equation*}
\frac{1}{(2\pi)^n}
\int_{\T^n}
\widetilde{f}_{a,b}(x) e^{-i k\cdot x}\,dx
=
\begin{cases}
{|k|^{-b}e^{i|k|^a}} & {\quad\text{if}\quad k\neq 0}
\\ 
{0} & {\quad\text{if}\quad k=0}.
\end{cases}
\end{equation*}
From this, we see the following.

\begin{lem}[{\cite[Lemma 6.1]{MT-2013}}]\label{sharpness-lem}
Let $0<a<1$, $0<b<n$,
and $\varphi \in \calS(\R^n)$.
For $t>0$, set 
\begin{equation}\label{sharpness-Wainger}
f_{a,b,t}(x)
=
\sum_{k \in \Z^n \setminus \{0\}}
e^{-t|k|}|k|^{-b}e^{i|k|^a}e^{i k\cdot x}\varphi(x).
\end{equation}
If $1\le q \le \infty$ and $b>n-an/2-n/q+an/q$, then
$\sup_{t>0}\|f_{a,b,t}\|_{L^q(\R^n)}<\infty$.
\end{lem}

Firstly we prove the following.

\begin{prop}\label{sharpnessAB}
Let 
$m \in (-\infty, 0]$, 
$q_j \in [1, \infty]$, $j=1, \dots, N$, 
and $r\in (0, \infty)$. 
Suppose all the multilinear operators $T_{\sigma}$ 
for $\sigma \in S^{\langle m \rangle }_{0,0} (\R^n, N)$ 
are bounded from 
$L^{q_1} \times \cdots \times L^{q_N}$ to $L^r$. 
Then 
$\sum_{j=1}^{N}1/q_j \ge 1/r$ and $m\le - (N-1)n/2$. 
\end{prop}

\begin{proof} 
We first prove the necessity of the condition 
$\sum_{j=1}^N 1/q_j \ge 1/r$. 
This is essentially due to 
\cite{GT} and \cite{grafakos 2014m}.

If the symbol $\sigma (x, \xi_1, \dots,\xi_N)$ is 
independent of $x$, then 
$\sigma$ is called a Fourier multiplier and 
$T_{\sigma}$ is called a multilinear Fourier multiplier operator. 
For multilinear Fourier multiplier operators, 
the following is known: 
if a nonzero Fourier multiplier operator $T_{\sigma}$
such that the support of the inverse Fourier transform of $\sigma$
is compact is bounded from $L^{q_1} \times \dots \times L^{q_N}$ to $L^r$, 
$0<q_j \le \infty$, and $0<r<\infty$,
then $\sum_{j=1}^N 1/q_j \ge 1/r$ 
(see \cite[Proposition 5]{GT} and 
\cite[Proposition 7.3.7]{grafakos 2014m}). 
Here, it should be remarked that
the case where some exponents $q_j$ are equal to infinity
was not discussed in those papers.
However, in our setting,
the definition of the boundedness of $T_{\sigma}(f_1,\dots,f_N)$
is restricted to functions $f_j \in \calS$.
Moreover, $\calS$ includes the set of smooth functions
with compact supports densely with respect to the $L^{\infty}$ norm.
Thus, the argument in the papers above works for such a case as well.

Now let $\sigma(\xi_1,\dots,\xi_N)$ be a 
nonzero function in $\calS((\R^n)^N)$
whose inverse Fourier transform has a compact support.
Then, since $\sigma(\xi_1,\dots,\xi_N)$ belongs to
$S^{\langle m \rangle }_{0,0}(\R^n, N)$ for 
any $m \le 0$,
the assumption of the proposition implies that 
$T_{\sigma}$ is bounded from $L^{q_1} \times \dots \times L^{q_N}$ 
to $L^r$.
Hence, by the fact mentioned above, 
we must have $\sum_{j=1}^N 1/q_j \ge 1/r$.

Next we prove the necessity of the condition 
$m\le -(N-1)n/2$. 
The argument below is based on the idea given in \cite[Proof of Lemma 6.3]{MT-2013}.
We first give a rather rough argument omitting necessary
limiting argument and we shall incorporate necessary details
at the last part of of proof.

By the closed graph theorem,
the assumption of the proposition implies that
there exists a positive integer $M$ such that
\begin{equation}\label{sharpness-m-assump}
\|T_\sigma\|_{L^{q_1} \times \dots \times L^{q_N} \to L^r}
\lesssim \max_{|\alpha| \le M}
\|\langle (\xi_1,\dots, \xi_N) \rangle^{-m}
\partial^{\alpha}
\sigma(x,\xi_1,\dots,\xi_N)\|_{L^{\infty}}
\end{equation}
for all 
$\sigma \in S^{\langle m \rangle }_{0,0}(\R^n, N)$,
where
$\partial^{\alpha}=\partial^{\alpha_0}_x\partial^{\alpha_1}_{\xi_1}
\dots \partial^{\alpha_N}_{\xi_N}$.

Let $\varphi, \widetilde{\varphi} \in \calS(\R^n)$
be such that
\begin{alignat*}{2}
&\mathrm{supp}\, \varphi \subset [-1/4,1/4]^n,
&\qquad
&|\calF^{-1}\varphi| \ge 1 \ \text{on} \ [-\pi,\pi]^n,
\\
&\mathrm{supp}\, \widetilde{\varphi} \subset [-1/2,1/2]^n,
&\qquad
&\widetilde{\varphi} = 1 \ \text{on} \ [-1/4,1/4]^n.
\end{alignat*}
We set
\begin{align*}
\sigma(x,\xi_1,\dots,\xi_N)
&=\sigma(\xi_1,\dots,\xi_N)
\\
&=\sum_{k_1,\dots,k_N\in \Z^n}
c_{k_1+\dots+k_N}
\langle (k_1,\dots,k_N) \rangle^{m}
\prod_{j=1}^N e^{-i|k_j|^{a_j}}
\widetilde{\varphi}(\xi_j-k_j),
\\
f_{a_j,b_j}(x)
&=\sum_{\ell_j \in \Z^n \setminus \{0\}}
|\ell_j|^{-b_j}e^{i|\ell_j|^{a_j}}e^{i \ell_j \cdot x}
\calF^{-1}\varphi(x),
\quad j=1,\dots,N,
\end{align*}
where $\{c_k\}_{k \in \Z^n}$
is a sequence satisfying $\sup_{k \in \Z^n}|c_k| \le 1$
that will be chosen later,
$0<a_j<1$,  
and $b_j=n-a_j n/2-n/q_j+a_j n/q_j+\epsilon_j$ 
with $\epsilon_j >0$. 
Here, 
we choose $\epsilon_j>0$ sufficiently small
so that $0<b_j<n$; this is possible since 
$b_j=n/2+(1-a_j)(n/2-n/q_j)+\epsilon_j$
and $|(1-a_j)(n/2-n/q_j)|<n/2$. 
Then
\begin{equation}\label{sharpness-m-pf-1}
|\partial^{\alpha_1}_{\xi_1}\dots\partial^{\alpha_N}_{\xi_N}
\sigma(\xi_1,\dots,\xi_N)|
\le C_{\alpha_1,\dots,\alpha_N} \langle (\xi_1,\dots,\xi_N) \rangle^m,
\end{equation}
where the constant is independent of
$\{c_k\}_{k \in \Z^n}$ 
satisfying $\sup_{k \in \Z^n}|c_k| \le 1$,
and by Lemma \ref{sharpness-lem} and Fatou's lemma,
\begin{equation}\label{sharpness-m-pf-2}
\prod_{j=1}^N \|f_{a_j,b_j}\|_{L^{q_j}} <\infty.
\end{equation}

Since
\begin{equation}\label{sharpness-m-s_0-pf}
\widehat{f_{a_j,b_j}}(\xi_j)=\sum_{\ell_j \in \Z^n \setminus \{0\}}
|\ell_j|^{-b_j}e^{i|\ell_j|^{a_j}}\varphi(\xi_j-\ell_j),
\end{equation}
it follows from the conditions on $\varphi, \widetilde{\varphi}$
that $T_{\sigma}(f_{a_1,b_1},\dots,f_{a_N,b_N})(x)$ can be written as
\begin{align*}
&\sum_{k_1,\dots,k_N\in \Z^n \setminus \{0\}}
c_{k_1+\dots+k_N}
\langle (k_1,\dots,k_N) \rangle^{m}
\left(\prod_{j=1}^N|k_j|^{-b_j}\right)
e^{i(k_1+\dots+k_N)\cdot x}
\calF^{-1}\varphi(x)^N
\\
&=\sum_{k \in \Z^n}c_k d_k e^{ik\cdot x} \calF^{-1}\varphi(x)^N
\end{align*}
with
\[
d_k=\sum_{\substack{k_1+\dots+k_N=k \\ k_1, \dots, k_N \neq 0}}
\langle (k_1,\dots,k_N) \rangle^{m}
\prod_{j=1}^N|k_j|^{-b_j}.
\]
Thus,
by the condition $|\calF^{-1}\varphi| \ge 1$
on $[-\pi,\pi]^n$,
\eqref{sharpness-m-pf-1},
and \eqref{sharpness-m-pf-2},
our assumption \eqref{sharpness-m-assump} implies
\begin{equation}\label{sharpness-m-pf-3}
\int_{[-\pi,\pi]^n}
\bigg|
\sum_{k \in \Z^n}c_k d_k e^{ik\cdot x}
\bigg|^r
dx
\lesssim 1.
\end{equation}
Here, it should be noticed that
the implicit constant 
in \eqref{sharpness-m-pf-3} 
depends on the quantity of 
\eqref{sharpness-m-pf-2}
but can be taken independent of  
$\{c_k\}_{k \in \Z^n}$
so far as $\sup_{k \in \Z^n}|c_k| \le 1$.

We choose 
$c_k = c_k (\omega)$  
to be identically distributed 
independent random variables on a probability 
space, each of which  
takes $+1$ and $-1$ with 
probability $1/2$. 
Then integrating 
over $\omega$ and using 
Khintchine's inequality, 
we have 
\begin{equation*}
\int 
\Big(\text{the left hand side of 
\eqref{sharpness-m-pf-3}} 
\Big)\, dP (\omega)
\approx 
\bigg(\sum_{k}
|d_k|^2 \bigg)^{r/2} 
\end{equation*}
(for Khintchine's inequality, 
see, e.g., \cite[Appendix C]{grafakos 2014c}). 
Hence,
\begin{equation}\label{sharpness-m-pf-4}
\bigg(\sum_{k \in \Z^n}|d_k|^2 \bigg)^{1/2}
\lesssim 1.
\end{equation}
Let $k \in \Z^n$ be such that $|k|$ is sufficiently large.
Each $d_k$ is equal to
\begin{equation*}
\sum_{\substack{k_2+\dots+k_N \neq k, \\ k_2, \dots, k_N \neq 0}}
\langle (k-k_2-\dots-k_N,k_2,\dots,k_N) \rangle^{m}
|k-k_2-\dots-k_N|^{-b_1} |k_2|^{-b_2} \dots |k_N|^{-b_N}
\end{equation*}
and this can be estimated from below by
\begin{align*}
&\sum_{0<|k_2|,\dots,|k_N| \le |k|/N}
\langle (k-k_2-\dots-k_N,k_2,\dots,k_N) \rangle^{m}
\\
&\qquad \qquad \times
|k-k_2-\dots-k_N|^{-b_1} |k_2|^{-b_2} \dots |k_N|^{-b_N}
\\
&\approx
|k|^{m-\sum_{j=1}^Nb_j+(N-1)n},
\end{align*}
where we used the fact $0<b_j<n$.
Then \eqref{sharpness-m-pf-4} yields
\begin{equation*}
m-\sum_{j=1}^Nb_j+(N-1)n
<-n/2.
\end{equation*}
Therefore,
by the arbitrariness of $0<a_j<1$ and $\epsilon_j>0$,
taking the limit as $a_j \to 1$ and $\epsilon_j \to 0$,
we must have $m-Nn/2+(N-1)n \le -n/2$,
namely $m \le -(N-1)n/2$.

The above argument is not entirely rigorous since,  
when we use 
Khintchine's inequality, we do not know a priori 
that $\sum |d_k|^{2}<\infty$ and since 
the functions 
$f_{a_j,b_j}$, $j=1,\dots,N$, are not in $\calS$. 
To get around these points, we replace $f_{a_j,b_j}$
by 
$f_{a_j,b_j,t}$ 
defined by \eqref{sharpness-Wainger}
with $\varphi$ replaced by $\calF^{-1}\varphi$. 
Then $f_{a_j,b_j,t}$ is a function in $\calS$ and 
Lemma \ref{sharpness-lem} gives
\begin{equation*}
\sup_{t>0}
\left(\prod_{j=1}^N\|f_{a_j,b_j,t}\|_{L^{q_j}}\right)<\infty.
\end{equation*}
If we define $\sigma$ in the same way as above, 
then we have 
\begin{equation*}
T_{\sigma}(f_{a_1,b_1,t}, \dots, f_{a_N,b_N,t})(x)
=\sum_{k \in \Z^n}c_k d_{k,t}
e^{i k \cdot x}\calF^{-1}{\varphi}(x)^N
\end{equation*}
with
\[
d_{k,t}=
\sum_{\substack{k_1+\dots+k_N=k \\ k_1, \dots, k_N \neq 0}}
\langle (k_1,\dots,k_N) \rangle^{m}
\prod_{j=1}^Ne^{-t|k_j|}|k_j|^{-b_j}, 
\]
which certainly satisfies 
$\sum_{k}|d_{k,t}|^2 <\infty$. 
Hence, by the argument as given above, 
we see that the estimate \eqref{sharpness-m-pf-4}
with $d_{k}$ replaced by $d_{k,t}$ holds 
with an implicit constant independent of $t>0$. 
Therefore, taking the limit as $t\to 0$, we obtain 
$\sum_{k}|d_{k}|^{2}<\infty$. 
The rest of the argument is the same as above.
\end{proof}

The next proposition shows that the condition 
on $s_0$ in Theorems \ref{main-thm-3} is sharp.

\begin{prop}\label{sharpnessC}
Let 
$m=-(N-1)n/2$, 
$q_1, \dots, q_N \in [1, \infty]$, 
$r\in (0, \infty]$, and 
$s_0, s_1, \dots, s_N \in [0, \infty)$. 
Suppose there exists a  
$t \in (0, \infty]$ such that 
the estimate 
\begin{equation}\label{sharpness-assump-s_0-t}
\begin{split}
&
\|T_{\sigma} 
\|_{ L^{q_1} \times \cdots \times L^{q_N} \to L^r}
\\
&\lesssim 
\left\|
2^{\boldsymbol{k}\cdot\boldsymbol{s}}
\big\|
\langle (\xi_1, \dots, \xi_N) \rangle^{-m} 
\Delta_{\boldsymbol{k}}\sigma(x,\xi_1,\dots,\xi_N)\big\|_{L^{\infty} ((\R^n)^{N+1}) }
\right\|_{\ell^t_{k_0, k_1, \dots, k_N} ((\Z^n)^{N+1}) }
\end{split}
\end{equation}
holds for all smooth functions $\sigma$ 
with the right hand side finite.
Then $s_0 \ge n/2$. 
\end{prop}

\begin{proof}
Firstly, we observe that 
it is sufficient to deduce the condition $s_0 \ge n/2$ under the 
assumption that \eqref{sharpness-assump-s_0-t} holds 
with $t=\infty$. 
In fact, once this is proved, 
then replacing  
$s_j$ by $s_j+\epsilon$, $\epsilon>0$,
$j=0,1,\dots,N$, 
we see that \eqref{sharpness-assump-s_0-t} with $t\in (0, \infty)$ 
implies 
$s_0+\epsilon \ge n/2$. 
Thus since $\epsilon >0$ is arbitrary, 
we must have $s_0 \ge n/2$.

Now, since our method below is similar to 
the one used in the proof of Proposition \ref{sharpnessAB}, 
we shall only give an argument omitting necessary
limiting argument. 
Suppose 
\eqref{sharpness-assump-s_0-t} holds with $t=\infty$. 
Take a function 
$\varphi \in \calS(\R^n)$ 
such that 
$\mathrm{supp}\, \varphi \subset [-1/2,1/2]^n$
and $\int_{\R^n} \varphi (\xi)^2\, d\xi \neq 0$ 
and set
\begin{align*}
&\sigma(x,\xi_1,\dots,\xi_N)
=\varphi(x)
e^{-ix \cdot (\xi_1+\dots+\xi_N)}
\\
&\qquad \qquad \qquad \qquad \times
\sum_{\ell_1,\dots,\ell_N \in \Z^n}
\langle (\ell_1,\dots,\ell_N) \rangle^{m-s_0}
\prod_{j=1}^N e^{-i|\ell_j|^{a_j}}
\varphi(\xi_j-\ell_j),
\\
&f_{a_j,b_j}(x)=\sum_{\ell_j \in \Z^n \setminus \{0\}}
|\ell_j|^{-b_j}e^{i|\ell_j|^{a_j}}e^{i \ell_j \cdot x}
\calF^{-1}{\varphi}(x),
\quad j=1,\dots,N,
\end{align*}
where $0<a_j<1$  
and $b_j=n-a_j n/2-n/q_j+a_j n/q_j+\epsilon_j$ 
with $\epsilon_j>0$. 
Here, we choose $\epsilon_j>0$ sufficiently small
so that $0<b_j<n$.
Let $L_j$ be a nonnegative integer
satisfying $L_j \ge s_j$ for $j=0,1,\dots,N$.
Since
\[
|\partial^{\alpha_0}_x
\partial^{\alpha_1}_{\xi_1} \dots \partial^{\alpha_N}_{\xi_N}
\sigma(x,\xi_1,\dots,\xi_N)|
\le
C_{\alpha_0,\alpha_1,\dots,\alpha_N}
\langle (\xi_1,\dots,\xi_N) \rangle^{m-s_0+|\alpha_0|},
\]
we see that
\[
|\Delta_{\boldsymbol{k}}\sigma(x,\xi_1,\dots,\xi_N)|
\lesssim
\begin{cases}
\langle (\xi_1,\dots,\xi_N) \rangle^{m-s_0}
2^{-k_1 L_1-\dots-k_N L_N},
\\
\langle (\xi_1,\dots,\xi_N) \rangle^{m-s_0+L_0}
2^{-k_0 L_0-k_1 L_1-\dots-k_N L_N}
\end{cases}
\]
(see \cite[Subsection 5.3]{KMT-arXiv}).
Thus,
taking $0 \le \theta_0 \le 1$ satisfying $s_0=L_0\theta_0$,
we have
\begin{align*}
|\Delta_{\boldsymbol{k}}\sigma(x,\xi_1,\dots,\xi_N)|
&=|\Delta_{\boldsymbol{k}}\sigma(x,\xi_1,\dots,\xi_N)|^{1-\theta_0}
|\Delta_{\boldsymbol{k}}\sigma(x,\xi_1,\dots,\xi_N)|^{\theta_0}
\\
&\lesssim
\left(\langle (\xi_1,\dots,\xi_N) \rangle^{m-s_0}
2^{-k_1L_1-\dots-k_NL_N}\right)^{1-\theta_0}
\\
&\qquad \times
\left(\langle (\xi_1,\dots,\xi_N) \rangle^{m-s_0+L_0}
2^{-k_0L_0-k_1L_1-\dots-k_NL_N}\right)^{\theta_0}
\\
&=\langle (\xi_1,\dots,\xi_N) \rangle^{m}
2^{-k_0s_0-k_1L_1-\dots-k_NL_N},
\end{align*}
which gives
\begin{equation}\label{sharpness-proof-s_0-1}
\sup_{\boldsymbol{k} \in (\N_0)^{N+1}}
2^{\boldsymbol{k}\cdot\boldsymbol{s}}
\|\langle (\xi_1,\dots,\xi_N) \rangle^{-m}
\Delta_{\boldsymbol{k}}\sigma(x,\xi_1,\dots,\xi_N)\|_{L^{\infty}}
<\infty.
\end{equation}
On the other hand,
by Lemma \ref{sharpness-lem},
\begin{equation}\label{sharpness-proof-s_0-2}
\prod_{j=1}^N
\|f_{a_j,b_j}\|_{L^{q_j}}
<\infty.
\end{equation}

Using the support condition of $\varphi$ and
\eqref{sharpness-m-s_0-pf},
we see that
\begin{align*}
T_{\sigma}(f_{a_1,b_1},\dots,f_{a_N,b_N})(x)
&=(2\pi )^{-Nn} 
\varphi(x)
\sum_{\ell_1,\dots,\ell_N\in \Z^n \setminus \{0\}}
\langle (\ell_1,\dots,\ell_N) \rangle^{m-s_0}
\\
&\qquad \times
\prod_{j=1}^N
|\ell_j|^{-b_j}
\int_{\R^n}\varphi(\xi_j-\ell_j)^2\, d\xi_j
\\
&=C\left(\sum_{\ell_1,\dots,\ell_N\in \Z^n \setminus \{0\}}
\langle (\ell_1,\dots,\ell_N) \rangle^{m-s_0}
\prod_{j=1}^N
|\ell_j|^{-b_j}\right)
\varphi(x).
\end{align*}
Hence the assumption \eqref{sharpness-assump-s_0-t} with 
$t=\infty$ together with \eqref{sharpness-proof-s_0-1}
and \eqref{sharpness-proof-s_0-2} 
implies
\begin{equation}\label{finite}
\sum_{\ell_1,\dots,\ell_N\in \Z^n \setminus \{0\}}
\langle (\ell_1,\dots,\ell_N) \rangle^{m-s_0}
\prod_{j=1}^N
|\ell_j|^{-b_j}<\infty.
\end{equation}
The last sum is estimated from below by
\begin{align*}
&\sum_{\ell_1 \in \Z^n \setminus \{0\}}
|\ell_1|^{-b_1}
\left(
\sum_{0<|\ell_2|, \dots, |\ell_N| \le |\ell_1|}
\langle (\ell_1,\dots,\ell_N) \rangle^{m-s_0}
\prod_{j=2}^N
|\ell_j|^{-b_j}\right)
\\
&\approx \sum_{\ell_1 \in \Z^n \setminus \{0\}}
|\ell_1|^{m-s_0-b_1+\sum_{j=2}^N (n-b_j)}
\end{align*}
Thus \eqref{finite} implies $m -s_0-b_1 +\sum_{j=2}^N (n-b_j)<-n$.
Therefore, by the arbitrariness of $0<a_j<1$ and $\epsilon_j>0$,
taking the limit as $a_j \to 1$, $\epsilon_j \to 0$, and 
$b_j \to n/2$, 
we obtain $s_0 \ge m + n/2 + (N-1)n/2 = n/2$.
\end{proof}

Finally, the following proposition shows that the conditions 
on $s_1, \dots, s_N$ in Theorem \ref{main-thm-3} are sharp.

\begin{prop}\label{sharpnessDE}
Let 
$m\in (-\infty, 0]$, 
$q_1, \dots, q_N \in [1, \infty]$, 
$r\in (0, \infty]$, 
and 
$s_0, s_1, \dots, s_N \in [0, \infty)$. 
Suppose there exists a  
$t \in (0, \infty]$ such that 
the estimate 
\begin{equation}\label{sharpness-assump-s_j-t}
\|T_{\sigma} 
\|_{ L^{q_1} \times \cdots \times L^{q_N} \to L^r}
\\
\lesssim 
\|
\sigma 
\|
_{ S^{\langle m \rangle }_{0,0} 
(s_0, s_1, \dots, s_N, t; \R^n, N) }
\end{equation}
holds for all smooth functions $\sigma$ 
with the right hand side finite. 
Then $s_j \ge n/2- n/q_j$, $j=1, \dots, N$, 
and $\sum_{j=1}^{N} s_j \ge \sum_{j=1}^{N} (n/2 - n/q_j) + n/r$. 
\end{prop}

\begin{proof}
Let $\psi_k \in \calS(\R^n)$, $k \ge 0$,
be the same as in Subsection \ref{subsectionMainTh},
but here we choose $\psi_0$ (in other words, $\phi$) to be
a real-valued radial function.
Set $\varphi(y)=\psi_0(2y)$.
Thus 
$\mathrm{supp}\, \varphi \subset 
\{ y\in \R^n : |y| \le 1\}$ 
and $\psi_k, \varphi$ are 
also real-valued radial functions.

We first show $s_j \ge n/2-n/q_j$. 
By symmetry, it is sufficient to consider the case $j=1$. 
Let $a$ be a positive integer and set
\begin{equation}\label{sharpness-s_1-def}
\begin{split}
&
\sigma_{a}(x,\xi_1,\dots,\xi_N)
=\sigma_{a}(\xi_1,\dots,\xi_N)
=\calF^{-1}\psi_a(\xi_1)
\prod_{j=2}^N\calF^{-1}\varphi(\xi_j), 
\\
&
f_{1, a}(x)=\psi_a (x),
\quad
f_2(x) = \cdots =f_N(x)=\varphi (x). 
\end{split}
\end{equation}
Observe that 
$\Delta_{\boldsymbol{k}}\sigma_a(\xi_1,\dots,\xi_N)$
is equal to
$\calF^{-1}[ \psi_{k_1} \psi_a] (\xi_1)
\prod_{j=2}^N\calF^{-1}\varphi(\xi_j)$
if $k_0=k_2=\dots=k_N=0$
and $|k_1 - a| \le 1$,
and 
$0$ otherwise.
Here we used the fact that $\psi_{k_0}(D_x)[1]$ is equal 
to $\psi_{0}(0)=1$ if $k_0=0$ and 
to $\psi (0)=0$ if $k_0 \ge 1$. 
Notice also  
$|\calF^{-1} [ \psi_{k_1} \psi_a](\xi_1)|
\lesssim 2^{an}\langle 2^{a}\xi_1 \rangle^{-L}$
if $|k_1 - a| \le 1$,
where $L$ can be chosen arbitrarily large. 
Thus we have
\begin{equation}\label{sharpness-proof-s_1-1}
\begin{split}
\|\sigma_a\|_{S^{\langle m \rangle }_{0,0} 
(\boldsymbol{s}, t; \R^n, N)}
&\lesssim
2^{as_1}
\left\|\langle(\xi_1,\dots,\xi_N) \rangle^ {- m}
2^{an}\langle 2^{a}\xi_1 \rangle^{-L}
\prod_{j=2}^N\langle \xi_j \rangle^{-L}
\right\|_{L^2 (\R^{nN})}
\\
&\lesssim 2^{a(s_1+n/2)},
\end{split}
\end{equation}
where the implicit constants may depend 
on $\boldsymbol{s}$ and $t$
but no on $a$. 
On the other hand, 
\begin{equation}\label{sharpness-proof-s_1-2}
\|f_{1,a}\|_{L^{q_1}} 
\prod_{j=2}^N\|f_j\|_{L^{q_i}}
=
\|\psi_{a}\|_{L^{q_1}} 
\prod_{j=2}^{N} 
\|\varphi\|_{L^{q_j}}
\approx 2^{an/q_1}.
\end{equation}
Since $\psi_{a}$ and $\varphi$ are 
radial functions, we have 
\begin{align*}
T_{\sigma_a}(f_{1,a},f_2,\dots,f_N)(x) 
&= 
(2\pi )^{-Nn}
\big(\psi_{a} \ast \psi_{a} \big) (x)
\prod_{j=2}^{N} \big(\varphi \ast \varphi\big) (x)
\\
&=
(2\pi )^{-Nn} 2^{a n} 
\big(\psi \ast \psi \big)( 2^{-a}x)
\prod_{j=2}^{N} \big(\varphi \ast \varphi\big) (x). 
\end{align*}
Moreover, since 
$\psi$ and $\varphi$ are radial real-valued functions,
we see that
$\psi \ast \psi(0)=
\|\psi\|_{L^2}^2 > 0$
and
$\varphi \ast \varphi(0)=
\|\varphi\|_{L^2}^2 > 0$.
Thus 
there exist
$C>0$ and $\delta>0$
such that
$|(\psi \ast\psi )(2^{-a}x)|\ge C$ and 
$|(\varphi \ast\varphi )(x) |\ge C$ 
for $|x| \le \delta$ and $a \ge 1$.
Hence 
\begin{equation}\label{xxx}
\|T_{\sigma_a}(f_{1,a},f_2,\dots,f_N)\|_{L^r}
\ge \|T_{\sigma_a}(f_{1,a},f_2,\dots,f_N)\|_{L^r(|x| \le \delta)}
\gtrsim 2^{an}. 
\end{equation}
Now 
the assumption \eqref{sharpness-assump-s_j-t} 
combined with the inequalities \eqref{sharpness-proof-s_1-1},  
\eqref{sharpness-proof-s_1-2}, 
and \eqref{xxx}, 
implies 
$2^{an} \lesssim 2^{a(s_1+n/2+n/q_1)}$. 
Since this holds for all $a \in \N$, we have $s_1 \ge n/2-n/q_1$.

We next show the sharpness of 
the condition $\sum_{j=1}^N s_j \ge \sum_{j=1}^N(n/2-n/q_j)+n/r$.
Since our argument is almost the same
as the preceeding case,
we only indicate the necessary modification.
Instead of \eqref{sharpness-s_1-def},
we set 
\begin{align*}
&\sigma_a (x,\xi_1,\dots,\xi_N)
=\sigma_a (\xi_1,\dots,\xi_N)
=\prod_{j=1}^N\calF^{-1}\psi_a(\xi_j),
\\
&f_{1, a}(x) =\cdots = f_{N,a} (x)=\psi_a (x) 
\end{align*}
where $a \in \N$.
In the same way as above,
we can prove
\begin{equation*}
\|\sigma_a\|_{S^{\langle m \rangle }_{0,0}
 (\boldsymbol{s},t; \R^n, N)}
\lesssim
2^{a\sum_{j=1}^N(s_j+n/2)},
\qquad
\prod_{j=1}^N\|f_{j,a}\|_{L^{q_j}}
\lesssim 2^{a \sum_{j=1}^N n/q_j }, 
\end{equation*}
where, in the former inequality,
the implicit constant may depend 
on $\boldsymbol{s}$ and $t$
but no on $a$. 
Since
$T_{\sigma_a}(f_{1,a},\dots,f_{N,a})(x)
= 
(2\pi )^{-Nn}
\left(2^{an} 
(\psi \ast \psi )(2^{-a}x)\right)^N$
in this case, we have 
\begin{equation*}
\|T_{\sigma_a}(f_{1,a},\dots,f_{N,a})\|_{L^r}
\approx \left(2^{an}\right)^N 2^{an/r}.
\end{equation*}
Therefore, the assumption \eqref{sharpness-assump-s_j-t} implies
\begin{equation*}
\left(2^{an}\right)^N 2^{an/r}
\lesssim 2^{a\left(\sum_{j=1}^N(s_j+n/2)+\sum_{j=1}^N n/q_j\right)}. 
\end{equation*}
Since this holds for all $a\in \N$, we have 
$\sum_{j=1}^N s_j \ge \sum_{j=1}^N(n/2-n/q_j)+n/r$.
\end{proof}


\end{document}